\documentclass{article}
\usepackage[preprint, nonatbib]{neurips_2024}

\usepackage[utf8]{inputenc} 
\usepackage[T1]{fontenc}    

\usepackage{hyperref}       
\usepackage{url}            
\usepackage{booktabs}       
\usepackage{amsfonts}       
\usepackage{nicefrac}       
\usepackage{microtype}      
\usepackage{xcolor}         

 \usepackage{amsmath, mathtools,amssymb}
 \usepackage{amsfonts}
 \usepackage{amsthm}
 \usepackage{bm}
 \usepackage{bbm}
 \usepackage{graphicx}
 \usepackage{booktabs}

\usepackage{nicefrac,xfrac} 
\usepackage{microtype}      
\usepackage{graphicx,subcaption,enumitem}
\usepackage[sort,comma,authoryear,round]{natbib}
\usepackage[ruled]{algorithm2e}
\usepackage{algorithmic}

\usepackage{thm-restate}

\usepackage[normalem]{ulem}

\usepackage{listings}
\usepackage{comment}

\usepackage{preamble}
\title{Randomized multi-class classification under system constraints: a unified approach via post-processing}
\author{%
  Evgenii Chzhen\\
  CNRS, Université Paris-Saclay\\
  \texttt{evgenii.chzhen@cnrs.fr} \\
  \And
  Mohamed Hebiri\\
  LAMA, Université Gustave Eiffel\\
  \texttt{mohamed.hebiri@univ-eiffel.fr}\\
  \And
  Gayane Taturyan\\
  IRT SystemX, Université Gustave Eiffel,\\
  Université Paul-Sabatier\\
  \texttt{gayane.taturyan@univ-eiffel.fr}
}

\begin{document}

\maketitle
\begin{abstract}
We study the problem of multi-class classification under system-level constraints expressible as linear functionals over randomized classifiers. We propose a post-processing approach that adjusts a given base classifier to satisfy general constraints without retraining. Our method formulates the problem as a linearly constrained stochastic program over randomized classifiers, and leverages entropic regularization and dual optimization techniques to construct a feasible solution. We provide finite-sample guarantees for the risk and constraint satisfaction for the final output of our algorithm under minimal assumptions. The framework accommodates a broad class of constraints, including fairness, abstention, and churn requirements.
\end{abstract}


\section{Introduction}

In many real-world settings, deploying a classifier involves more than just optimizing predictive accuracy. The predictions must often satisfy system-level constraints, such as fairness with respect to sensitive attributes, a limited rate of abstention, or consistency with previously deployed models.

A common strategy to enforce such constraints is to modify the training procedure itself—either through regularization or custom loss functions~\citep{Agarwal_Beygelzimer_Dubik_Langford_Wallach18,agarwal2019fair,charoenphakdee2021classification,oneto2020general,cao2022generalizing,cotter2019optimization}. However, these in-processing methods can be complex to implement, may require retraining from scratch, and often rely on strong assumptions about the underlying distribution or model architecture.

In this work, we take a different approach: we propose a post-processing method that transforms the predictions of an existing base classifier into a new randomized classifier that satisfies a collection of expectation-based constraints. This method can be applied in a black-box fashion, without requiring access to the internals of the original classifier, and is compatible with any existing estimation procedure.

Our framework is based on a linear programming formulation in which classifiers are represented as randomized prediction functions. The optimization objective is to minimize the expected loss under multiple expected constraint violations. To solve this problem, we rely on a dual formulation with entropic regularization, which leads to an explicit representation of the solution and enables the use of stochastic optimization techniques. We provide finite-sample guarantees for the resulting classifier, which hold under minimal assumptions on the data distribution. The combination of the above is one of the main distinctive feature of our approach compared to the previously available methodologies~\citep{celis2019classification,narasimhan2024consistent,narasimhan18a}.

The proposed method is flexible enough to handle a variety of constraint types, including fairness notions (such as demographic parity and equalized odds), rejection constraints, and churn control. It also naturally supports combinations of such constraints. Throughout the paper, we illustrate how different classification problems fit into our framework, and we provide theoretical guarantees that quantify constraint satisfaction and approximate risk minimization properties.

\paragraph{Contributions}
Our contribution is an end-to-end, theory-driven post-processing algorithm that is easy to implement, broadly applicable, and supported by statistical guarantees. It complements existing methods in the literature~\citep{celis2019classification,narasimhan2024consistent} and offers a practical alternative for constrained classification in complex environments.

\paragraph{Organization}
The paper is organized as follows. In Section~\ref{sec:setup}, we introduce the problem setup, presenting the general framework and motivating it through illustrative examples. Section~\ref{sec:methodology} details our methodology, emphasizing both its computational and statistical advantages. In Section~\ref{sec:algorithm}, we describe the proposed algorithm, which is designed to leverage an arbitrary black-box stochastic optimization procedure. We also instantiate it with \texttt{SGD3} of~\cite{allenzhu2021make} to provide a concrete understanding of the resulting performance bound as well as discuss an extension to set-valued classification. The proofs of the theoretical contributions, as well as a short numerical study are provided in the appendix.
\paragraph{Notation} Let us present the notation that is used throughout the paper. For a positive integer $K$, we write $[K]$ to denote $\{1, \ldots, K\}$. For every $\beta > 0, m \in \bbN$, and $\bw = (w_1, \ldots, w_m)^\top \in \bbR^m$, we denote by $\lse : \bbR^m \to \bbR$ and $\bsigma = (\sigma_1, \ldots, \sigma_m) : \bbR^m \to \bbR^m$
the log-sum-exp and the softmax functions respectively, defined as
\begin{align*}
    \lse (\bw) = \beta^{-1}\log\big(\sum_{j = 1}^m \exp(\beta w_j)\big) \quad \text{and} \quad \sigma_j(\bw) = {\exp(w_j)}/\big({\sum_{i = 1}^m \exp(w_i)}\big) \enspace.
\end{align*}
For any $a  \in \bbR$ and $\bw \in \bbR^m$  we set $(a)_+ = \max\{a , 0\}$ and $(\bw)_+ = ((w_1)_+, \ldots, (w_m)_+)^\top$. For $q\in [1,+\infty]$, the $\ell_q$-norm of a vector is denoted by $\norm{\cdot }_{q}$ and in particular, for $q=2$ we just write $\norm{\cdot}$ for the Euclidean norm.
The subordinate norm of an $m \times n$ matrix $\bA$ is defined as $\norm{\bA}_{p\rightarrow q} = \sup\left\{\norm{\bA\bw}_{q}:\norm{\bw}_p=1 \text{ and }\bw\in\bbR^n\right\}$, and we write $\| A \|_F = \norm{\bA}_{2\rightarrow 2}$ for the Frobenius norm of $A$.
Finally, the symbol $\independent $ stands for the statistical independence between two random variables.

\section{Problem setup}
\label{sec:setup}
In this work we consider a problem of multi-class classification. Let $\cX$ be the space of features and $[K]$ be the space of possible labels. 
Let $(\bX, Y) \in \cX \times [K]$ be the feature, label pair, following some unknown distribution $\Prob$.
We let $\cA$ stand for a finite set of possible values for prediction.
For example, in the simplest case, $\cA = [K]$ and other examples will be given in the next section.
To quantify a loss of issuing a given prediction for a given feature $\bx \in \cX$, we introduce loss function $\ell : \cX \times \cA \mapsto \bbR$ as well as $M \geq 0$ different cost functions $c_j : \cX \times \cA \mapsto \bbR$, which are \emph{not} necessarily known beforehand.
\begin{definition}[Randomized prediction]
    A randomized prediction function is a Markov kernel $\pi : \cX \mapsto \Delta(\cA)$ so that for any $\bx \in \cX$, $\pi(\cdot \mid \bx)$ is a probability distribution on $\cA$.
For any such prediction $\pi$, introduce $\Prob^\pi$ as a distribution of $(\bX, Y, \hY^\pi)$ so that
\[(\bX, Y) \sim \Prob,
\qquad 
\hY^\pi \mid \bX \sim \pi(\cdot \mid \bX)
\quad \text{ and }\quad
(\hY^\pi \independent Y) \mid \bX\,.
\]
We also write $\Exp^\pi[\cdot]$ to denote the expectation with respect to $\Prob^\pi$.

\end{definition}

Based on the previous definition, $\hY^{\pi}$ should be interpreted as the prediction.
In this work, we study, from a statistical point of view, the following optimization problem
\begin{align}
    \label{eq:optimal}
    \pi^\star \in \argmin_{\pi} \enscond{\Exp^\pi[\ell(\bX, \hY^\pi)]}{\Exp^\pi[c_j(\bX, \hY^\pi)] \leq 0 \quad \forall j \in [M]}\enspace.
\end{align}

\paragraph{On the risk, constraints, and role of labels.} Looking at~\eqref{eq:optimal}, it is evident that the labels $Y$ do not explicitly enter into the formulation of the problem, which is counter-intuitive. We note however, that as we do not assume that neither the loss, nor the constraint functions are known, the dependency on $Y$ can and will be hidden inside the two. For example, imagine we are given a loss function $L : \cY \times \cA \mapsto \bbR_+$, quantifying the (classical) loss of prediction $\hat{y}$ for the true value $y$. Then, setting
\begin{align*}
    \ell(\bx, \hat{y}) \eqdef \Exp\left[L(Y, \hat{y}) \mid \bX = \bx\right]\enspace,
\end{align*}
we would recover the setup that we consider in this paper. To give a concrete example, set $L(y, \hat{y}) = \ind{y \neq \hat{y}}$, then
\begin{align*}
    \ell(\bx, \hat{y}) = \sum_{y \in \cY} p_y(\bx)\ind{\hat{y} \neq y}\enspace,
\end{align*}
with $p_y(\bX) \eqdef \Prob(Y = y \mid \bX)$. While the loss $L$ was known, the loss $\ell$ is not and it should be estimated from data. In a typical problem loss and/or constraint functions will depend on the conditional distributions of labels (or other relevant quantities). Once these distributions are estimated, one can plug them into the corresponding loss and constraint functions.

Returning to the previous example, after constructing an estimator $\hat{p}_y(\cdot)$ for $p_{y}(\cdot)$---which can be done in various ways---one can define an estimated loss as
\[
\hat{\ell}(\bx, \hat{y}) = \sum_{y \in \cY} \hat{p}_y(\bx) \ind{\hat{y} \neq y}\enspace.
\]
Our theoretical guarantees will explicitly depend on the quality of the estimated loss and constraint function, as well as they will exhibit explicit cost for post-processing. This is precisely in this sense that we position our contribution within the post-processing approaches.

\subsection{Examples of problems}
\label{sub:examples_of_problems}
Let us provide several examples of concrete multi-class classification problems that fit into the framework described in the previous section. In the examples below we will stick to the simplest case of $\cA = [K]$ and write $\hY^{\pi}$ to denote the prediction.

\paragraph{Standard multi-class classification.} In the standard setup, \(\cA = [K]\), there are no costs, and the goal is to find a classifier that minimizes the misclassification risk:
\begin{align*}
\min_{\pi} \Prob^\pi(Y \neq \hat{Y}^\pi)\enspace.
\end{align*}
A direct computation shows that this objective can be expressed within our framework as follows:
\begin{align*}
\Prob^\pi(Y \neq \hat{Y}^\pi) = \Exp\left[\sum_{y \in \cY}(1 - p_k(\bX))\pi(y \mid \bX)\right] = \Exp^\pi[\ell(\bX, \hat{Y}^\pi)]\enspace,
\end{align*}
where \(\ell(\bx, \hat{y}) = 1 - p_{\hat{y}}(\bx)\).

Multi-class classification, including the binary case, has been extensively studied through various statistical approaches.

\paragraph{Classification with reject option.} Classification with reject option, first introduced by \citet{Chow57}, is a variation of classical classification problem, where the model is allowed to abstain from making prediction if it is uncertain about said prediction. To put it in our framework, we set $\cA = [K] \cup \{r\}$, where the prediction $r$ is interpreted as an abstention from prediction. Depending on the concrete application, we can consider two frameworks:
\begin{enumerate}
    \item Controlled rejection:
    \begin{align*}
        \min_{\pi}\enscond{\Prob^\pi(Y \neq \hat{Y}^{\pi}, \hat{Y}^{\pi} \neq r)}{\Prob^{\pi}(\hat{Y}^{\pi} = r) \leq \alpha}
    \end{align*}
    for some $\alpha \in (0, 1)$. In this framework, one targets minimization of misclassification when the predictor is not rejecting, while controlling the rate of rejection by some fixed parameter $\alpha$.

    \item Controlled error:
    \begin{align*}
        \min_{\pi}\enscond{\Prob^{\pi}(\hat{Y}^{\pi} = r)}{\Prob^\pi(Y \neq \hat{Y}^{\pi}, \hat{Y}^{\pi} \neq r) \leq \delta}\enspace,
    \end{align*}
    for some $\delta \in (0, 1)$. This framework is in some sense dual of the previous one. Instead of controlling the rejection rate, one fixes a given desired accuracy by $\delta$ and then seeks for a prediction function that minimizes the probability of rejection among such predictions.
\end{enumerate}
To put both frameworks in our setup, it suffices to notice that
\begin{align*}
    &\Prob^\pi(Y \neq \hat{Y}^{\pi}, \hat{Y}^{\pi} \neq r) = \Exp\left[\sum_{y \in \cY}(1 - p_{y}(\bX))\pi(y \mid \bX)\right] = \Exp^\pi[\ell(\bX, \hat{Y}^\pi)]\\
    &\Prob^{\pi}(\hat{Y}^\pi = r) = \Exp[\pi(r \mid \bX)] = \Exp[c(\bX, \hat{Y}^\pi)]\enspace,
\end{align*}
where again $\ell(\bx, \hat{y}) = 1 - p_{\hat{y}}(\bx)$ and $c(\bx, \hat{y}) = \ind{\hat{y} = r}$. Thus, either of the two constraint frameworks perfectly fits the setup of Eq.~\eqref{eq:optimal}. Indeed, it suffices to inverse the role of the loss and the constraints.

A statistical analyses of a penalized version of the problem was considered by~\cite{herbei2006classification,bartlett2008classification,yuan2010classification,wegkamp2007lasso}. A general statistical analysis of the first framework was done by~\cite{denis2020consistency}, while the second framework was studied in details by~\cite{lei2014}.

\paragraph{Classification under demographic parity constraint.}
In the literature of algorithmic fairness, it is often assumed that the features $\bX$ are composed of nominally unprotected attributes $\bZ$ and nominally protected attributes $S$. One possible goal is to learn a classifier whose prediction is independent from the sensitive attribute $S$.
More specifically, we assume that $\bX = (\bZ, S) \in \cZ \times \cS$ with $|\cS| < \infty$.
We let $\bW$ stand either for $\bZ$ or for $(\bZ, S)$ depending on the availability of $S$ at the inference time.

A typical objective in this literature is the following optimization problem
\begin{align*}
    \min_{\pi}\enscond{\Prob^\pi(Y \neq \hat{Y}^\pi)}{|\Prob^\pi(\hat{Y}^\pi = y \mid S = s) - \Prob^\pi(\hat{Y}^\pi = y)| \leq \varepsilon_s \,\,\forall (s, y) \in \cS \times \cY}\enspace.
\end{align*}
Intuitively, in the above problem, one looks for the most accurate classifier which is (approximately) independent from the sensitive attribute $S$.
Define
\begin{align*}
    p_y(\bw) \eqdef \Prob(Y = y \mid \bW = \bw) \quad \text{and} \quad \tau_s(\bw) \eqdef \Prob(S = s \mid \bW = \bw)\enspace.
\end{align*}
First of all, we again, as in the previous examples, notice that
\begin{align*}
    \Prob^\pi(Y \neq \hat{Y}^\pi) = \Exp^\pi[\ell(\bW, \hat{Y}^\pi)]\enspace,
\end{align*}
for $\ell(\bw, \hat{y}) = 1 - p_{\hat{y}}(\bw)$.
Secondly, we observe that the constraints in the above optimization problem are equivalent to the following $M = 2\cdot|\cS|\cdot|\cY|$ cost functions
\begin{align*}
    c_{(s, y)}^+(\bx, \hat{y}) &= \left(\frac{\tau_s(\bw)}{\Prob(S = s)} - 1\right)\ind{\hat{y} = y} - \varepsilon_s\enspace,\\
    c_{(s, y)}^-(\bx, \hat{y}) &= \left(1 - \frac{\tau_s(\bw)}{\Prob(S = s)}\right)\ind{\hat{y} = y} - \varepsilon_s\enspace.
\end{align*}
The combination of the two shows that this problem is well suited for the considered framework. Furthermore, it is important to point out that in the setup of $\bw = (\bz, s)$, one has $\tau_{s'}(\bz, s) = \ind{s = s'}$. Thus, in this setup the constraint function is known beforehand up to, potentially unknown, but easy to estimate $\Prob(S = s)$ for $s \in \cS$.

\paragraph{Classification under equal odds constraints.}
Still in the context of fairness from the previous example and the same notation, we consider
\begin{align*}
    \min_{\pi}\enscond{\Prob^\pi(Y \neq \hat{Y}^\pi)}{|\Prob^\pi(\hat{Y}^\pi = y \mid (S, Y) = (s, y')) - \Prob^\pi(\hat{Y}^\pi = y \mid Y = y')| \leq \varepsilon_{(s, y')} \,\,\forall (s, y, y')}\enspace.
\end{align*}
Intuitively, in the above problem, one looks for the most accurate classifier which is (approximately) independent from the sensitive attribute $S$, conditionally on the true outcome $Y$.
The risk is handled as in all the previous examples, while the constraints now correspond to the following cost $2\cdot|\cS| \cdot |\cY|^2$ functions
\begin{align*}
    c_{(y, s, y')}^+(\bw, \hat{y}) &= \left(\frac{\Prob((S, Y) = (s, y') \mid \bW = \bw)}{\Prob((S, Y) = (s, y'))} - \frac{p_{y'}(\bw)}{\Prob(Y = y')}\right)\ind{\hat{y} = y} - \varepsilon_{(s, y')}\enspace,\\
    c_{(y, s, y')}^-(\bw, \hat{y}) &= \left(\frac{p_{y'}(\bw)}{\Prob(Y = y')} - \frac{\Prob((S, Y) = (s, y') \mid \bW = \bw)}{\Prob((S, Y) = (s, y'))}\right)\ind{\hat{y} = y} - \varepsilon_{(s, y')} \enspace.
\end{align*}

Demographic parity and equal odds constraints for both \(\bW = \bZ\) and \(\bW = (\bZ, S)\) have been explored in numerous studies using a variety of approaches. Broadly, two main types of algorithms have been investigated: in-processing methods~\citep{Agarwal_Beygelzimer_Dubik_Langford_Wallach18,oneto2020general} and post-processing methods~\citep{schreuder2021,pmlr-v206-gaucher23a,xian2023fair,denis2021,zeng2022fair}. Our approach falls within the post-processing paradigm, introducing minimal data-driven modifications to a given estimator to ensure compliance with fairness constraints.

\paragraph{Classification under churn constraint.}
Classification with churn constraints aims at classifying instances while limiting the number of changes (or “churn”) between the new and previous classifications~\citep{cotter2019optimization,narasimhan2024consistent,narasimhan18a,narasimhan2020pairwise}. This is particularly relevant in settings where frequent changes in predictions can have undesirable consequences. More specifically, in the setting of the standard $0/1$-risk, one wishes to solve
\begin{align*}
    \min_{\pi}\enscond{\Prob^{\pi}(Y \neq \hat{Y}^\pi)}{\Prob^\pi(\hat{Y}^\pi \neq g(\bX)) \leq \alpha}\enspace,
\end{align*}
where $g$ is some fixed classifier and $\alpha \in (0, 1)$ quantifies the desired level of disagreement with $g$. Again, using the same expression for the risk as before and noticing that
\begin{align*}
    \Prob^\pi(\hat{Y}^\pi \neq g(\bX)) = \Exp^\pi\left[\indin{\hat{Y}^\pi \neq g(\bX)}\right]\enspace,
\end{align*}
we are able to accommodate this example in our framework.

\paragraph{Combinations of the above constraint.} In a more realistic scenario, one may need to enforce various constraints. For instance, as discussed in~\citep{schreuder2021}, a practitioner might seek to satisfy a demographic parity constraint while maintaining a controlled rejection rate. Any such combination—provided feasibility is addressed—can be seamlessly integrated into our framework through appropriate choices of the prediction set \(\cA\), loss function $\ell$, and constraints $c_j$.

\subsection{Linear program representation}
We return to the general framework.
It is no surprise that by considering the randomized classifiers in~\eqref{eq:optimal}, the resulting problem becomes that of linear programming. In particular, both risk $\risk(\pi)$ and costs $\cC_j(\pi)$ can be expressed as
\begin{align}
    \risk(\pi)
    &= \Exp_\bX \Big[\sum_{a\in \cA}\ell(\bX, a)\pi(a \mid \bX)\Big] \quad \text{and} \label{eq:risk} \\
    \cC_j(\pi)
    &= \Exp_\bX \Big[\sum_{a\in \cA}c_j(\bX, a)\pi(a \mid \bX)\Big] \quad \text{for} \quad j \in [M]\enspace. \label{eq:constraint}
\end{align}
For an $\bx \in \bbR^d$ we introduce a vector $\bL(\bx) \eqdef (\ell(\bx,a))_{a \in \cA}\in \bbR^{|\cA|}$, and a matrix $\bC(\bx) \in \bbR^{M \times |\cA|}$ defined as
\begin{align*}
    (\bC(\bx))_{ja} = c_{j}(\bx, a)\enspace,
\end{align*}
for $j \in [M]$ and $a \in \cA$. With that notation at hand, the original problem in~\eqref{eq:optimal} can be written as\footnote{When the symbols $\leq$ or $<$ are applied to vectors, the inequalities are to be understood component-wise.}
\begin{align}
    \label{eq:optimal_LP}
    \min\enscond{\Exp\scalar{\bL(\bX)}{\bpi(\bX)}}{\Exp\left[\bC(\bX)\bpi(\bX)\right] \leq 0}\enspace,
\end{align}
where we denoted by $\bpi(\bx)$ the vector in $\bbR^{|\cA|}$ composed of $\pi(a \mid \bx)$ for $a \in \cA$.

Our approach will be strongly dependent on some strong duality arguments for the above linear program. More concretely, our analysis will reply on the boundedness of the optimal dual variables associated to~\eqref{eq:optimal_LP}. Thus, we need to introduce some constraint qualifications conditions. One such condition that is also rather simple to use is Slater's condition, which we put as an assumption on~\eqref{eq:optimal_LP}.
\begin{assumption}[{Slater's condition}]
    \label{ass:slater}
    There exists $\bpi$ \st $\Exp\left[\bC(\bX)\bpi(\bX)\right] < 0 $.
\end{assumption}

Slater's condition allows to prove the following generalized version of Neyman-Pearson lemma.
\begin{lemma}[Generalized Neyman-Pearson lemma]
    \label{lem:NP}
    Let Assumption~\ref{ass:slater} be satisfied and assume that
    \begin{align*}
        \Exp\|\bL(\bX)\| < \infty,\qquad \Exp\|\bC(\bX)\|_{F} < \infty\enspace.
    \end{align*}
    Then, there exists a $\bpi^\star$ solution of~\eqref{eq:optimal_LP} and $\blambda^\star \in \bbR^{M}_+$ such that
    \begin{align*}
    \begin{cases}
        &\displaystyle\blambda^\star \in \argmax_{\blambda \geq 0} \Exp[\min_{a \in \cA}(\bL(\bX) + \bC^\top(\bX)\blambda)_a]\\
        &\displaystyle\bpi^\star \in \argmin_{\bpi} \Exp[\scalar{\bL(\bX) + \bC^\top(\bX)\blambda^\star}{\bpi(\bX)}]\enspace
    \end{cases}
    \quad\text{and}\quad
    \begin{cases}
        \scalar{\blambda^\star}{\Exp[\bC(\bX)\bpi^\star(\bX)]} = 0\\
        \Exp[\bC(\bX)\bpi^\star(\bX)] \leq 0
    \end{cases}\,.
    \end{align*}
    Moreover, the support of the optimal randomized classifier satisfies
    \begin{align*}
        \mathrm{supp}(\pi^*(\cdot \mid \bx)) = \argmin_{a \in \cA}\ens{\ell(\bx, a) + \sum_{j = 1}^M c_j(\bx, a)\lambda^\star_j} \qquad \text{a.e.-}\Prob_{\bX}\,.
    \end{align*}
\end{lemma}
\begin{proof}
    First let us show that the infimum is achieved. We note that
    \begin{align*}
        \inf\enscond{\Exp\scalar{\bL(\bX)}{\bpi(\bX)} \eqdef H(\bpi)}{\Exp\left[\bC(\bX)\bpi(\bX)\right] \leq 0} \geq 0\enspace.
    \end{align*}
    For completeness, we define spaces on which we work: we assume that $\Prob_{\bX}$ is a $\sigma$-finite probability measure and define
    \begin{align*}
        L^{1}(\cX,\Prob_{\bX};\mathbb{R}^{|\cA|})
      &\eqdef \Bigl\{ f : \cX \to \mathbb{R}^{|\cA|} \; \;:\;
                  \Exp \|f(\bX)\| < \infty \Bigr\}\\
        L^{\infty}(\cX,\Prob_{\bX};\mathbb{R}^{|\cA|})
      &\eqdef \Bigl\{ f : \cX \to \mathbb{R}^{|\cA|} \; \;:\;
                  {\displaystyle\operatorname*{ess\,sup}_{\bx \in \cX}}\|f(\bx)\| < \infty \Bigr\} \enspace.
    \end{align*}
    Taking minimizing sequence $\bpi_n \in L^{\infty}(\cX,\mu;\mathbb{R}^{|\cA|})$ for the above problem and invoking Banach-Alaoglu, we extract a weak*-convergent subsequence. To conclude we note that $\bpi \mapsto \Exp\scalar{\bL(\bX)}{\bpi(\bX)}$ is a weak*-continuous functional by our assumption on $\bL \in L^{1}(\cX,\mu;\mathbb{R}^{|\cA|})$. Thus, the infimum is achieved by Weierstrass theorem.

    The claim follows from the Lagrange duality that can be obtained for example from~\cite[Theorem 1, p. 217 and Corollary 1, p. 219]{luenberger1997optimization}.
\end{proof}
The goal of this work is to develop an algorithmic approach to the above lemma, enabling finite-sample guarantees in a black-box manner.

In the case of testing two simple hypotheses, it is relatively straightforward to show that these conditions lead to the well-known likelihood ratio test. This follows from the fact that the underlying system of equations can be solved explicitly. However, when a closed-form solution is unavailable, alternative approaches are required—one of which is developed in this work.

We also note that~\cite[Theorem 3.2 from][]{celis2019classification} is very similar to Lemma~\ref{lem:NP}, however the former is slightly problematic from the theoretical perspective. The issue is discussed in Section~\ref{sub:related_works}, where a detailed comparison is provided.

\section{Proposed methodology}
\label{sec:methodology}
To address problem~\eqref{eq:optimal}, we build upon the approach proposed in~\citep{NEURIPS2024_d5c3ecf3}. Our high-level strategy is grounded in convex optimization, drawing inspiration from Nesterov’s smoothing technique~\citep{nesterov2005}, as well as from the optimal transport literature, where entropic regularization is employed to enhance numerical stability~\citep{peyre2019computational}. Our objective is to formulate a smoothed dual version of problem~\eqref{eq:optimal} and, by leveraging the primal-dual relationship, demonstrate that an approximate minimizer of the smoothed dual leads to a statistically sound randomized classifier.
To this end, we introduce the entropic regularized risk \[\risk_{\beta}(\pi) = \risk(\pi) + \frac{1}{\beta}\Exp\bigg[\underbrace{\sum_{a \in \cA} \pi(a \mid \bX)\log(\pi(a \mid \bX))}_{\eqdef \Psi(\bpi(\cdot \mid \bX))}\bigg]\enspace.\]
Instead of \eqref{eq:optimal}, we consider the following problem\footnote{To avoid clutter, we slightly abuse notation and use $\risk_{\beta}(\cdot)$ interchangeably for the kernel $\pi$ or the induced vector $\bpi$}
\begin{align}
    \label{eq:optimal_entropic_init}
    \min_{\bpi} \enscond{\risk_{\beta}(\bpi)}{\cC_j(\bpi) \leq 0, \quad \forall j \in [M]}\enspace.
\end{align}
An appealing aspect of the problem in~\eqref{eq:optimal_entropic_init} is that its solution can be explicitly expressed as a function of optimal dual variables, which are obtained by solving a stochastic convex program with a Lipschitz gradient. This result is summarized in the following lemma.
\begin{lemma}
    \label{lemma:relaxation+smoothing}
    For any $\beta > 0$, define $\blambda^\star \in \bbR^M_+$ as a solution of
    \begin{align}
    \label{min-lse}
        \min_{\blambda \geq 0} \left\{ F(\blambda) \eqdef \Exp_\bX\left[\lse\left(
        -\bL(\bX) - \bC(\bX)^\top\blambda\right)\right]\right\}  \enspace.
    \end{align}
Then, under Assumption~\ref{ass:slater}, ~\eqref{eq:optimal_entropic_init} admits a solution in the form
\begin{align}
    \label{eq:optimal_entropic}
    \pi_{\blambda^\star}(a \mid \bx) \propto \exp\left(\beta\left(
    -\bL(\bx) - \bC(\bx)^\top\blambda^\star\right)_a\right) \text{ for } a \in \cA \enspace.
\end{align}
\end{lemma}
\begin{proof}[\textbf{Proof of Lemma}~\ref{lemma:relaxation+smoothing}]
Let us introduce the Lagrangian for the problem in \eqref{eq:optimal_entropic_init}, as:
\begin{align}
    \label{eq:lagrangian}
    \class{L}(\bpi, \blambda)
    &= \risk_{\beta}(\bpi) + \sum_{j \in [M]}\lambda_j\cC_j(\bpi) \nonumber \\
    &= \Exp\left[\sum_{a \in \cA} \left(\ell(\bX, a) + \scalar{\blambda}{\bc(\bX, a)}\right)\pi(a \mid X) + \frac{1}{\beta}\Psi(\bpi(\cdot \mid \bX))\right]\enspace.
\end{align}
Using Sion's minmax theorem, we get the following strong duality
\begin{align*}
     \min_{\bpi} \enscond{\risk_{\beta}(\bpi)}{\cC_j(\bpi) \leq 0, \quad \forall j \in [M]} = \min_{\bpi} \max_{\blambda \geq 0} \class{L}(\bpi, \blambda)  =  \max_{\blambda \geq 0} \min_{\bpi}\class{L}(\bpi, \blambda) \enspace.
\end{align*}
Using the variational representation of $\lse$ that is recalled in Lemma~\ref{lemma:lse_variational} of appendix, we solve the inner minimization problem appearing on the right hand side of the above display. In particular, it holds that
\begin{align*}
    \min_{\bpi} \class{L}(\bpi, \blambda) &= - \max_{\bpi} \left\{\Exp\left[\sum_{a \in \cA} \left(-\ell(\bX, a) - \scalar{\blambda}{\bc(\bX, a)}\right)\pi(a \mid X) - \frac{1}{\beta}\Psi(\bpi(\cdot \mid \bX))\right]\right\} \\
    &= - \Exp\left[\lse\left(\left(
    -\ell(\bX, a) - \scalar{\blambda}{\bc(\bX, a)}\right)
    _{a\in\cA}\right)\right] \enspace,
\end{align*}
where optimum for every $\blambda \geq 0$ is achieved at
\[
\pi_{\blambda}(a \mid \bx) \propto \exp\left(\beta\left(-\ell(\bx, a)-\scalar{\blambda}{\bc(\bx,a)}\right)\right), \quad a \in \cA \enspace.
\]
Finally, we conclude that
\begin{equation*}
     \min_{\bpi} \enscond{\risk_{\beta}(\bpi)}{\cC_j(\bpi) \leq 0, \quad \forall j \in [M]} = \max_{\blambda \geq 0} \{-F(\blambda)\} = \risk_\beta (\bpi_{\blambda^*})\enspace.\qedhere
\end{equation*}
\end{proof}
Assumption~\ref{ass:slater} is essentially used here to establish the strong duality and ensure that the minimizer of~\eqref{min-lse} exists.

Let us also summarize some important properties of the derived optimal entropic-regularized prediction function. We show that $\bpi_{\blambda^\star}$ is in fact feasible for the initial problem of interest in~\eqref{eq:optimal}.
\begin{lemma}[{Feasibility}]
    \label{lemma:constraints}
    Let $\bpi_{\blambda^\star}$ be defined as in Lemma~\ref{lemma:relaxation+smoothing}. Then, $\cC_j(\bpi_{\blambda^\star}) \leq 0$ for all $j \in [M]$.
\end{lemma}
\begin{proof}[\textbf{Proof of Lemma}~\ref{lemma:constraints}]
Our goal is to show that $\bpi^\star$ satisfies the required constraints. We are going to rely on Lemma~\ref{lem:KKT}. Indeed, notice that since $\gamma_j \geq 0$ for any $j \in [M]$, we have
\begin{equation*}
    \cC_j(\bpi^*) = \Exp\left[\sum_{a\in \cA}c_j(\bX, a)\pi^*(a \mid \bX)\right] = - \gamma_j \leq 0 \enspace.\qedhere
\end{equation*}
\end{proof}

Moreover, we show that the risk of $\bpi_{\blambda^\star}$ is controlled by entropic-regularization parameter $\beta$.
\begin{lemma}[{Risk gain}]
    \label{lemma:risk}
    Let $\bpi_{\blambda^\star}$ be defined in Lemma~\ref{lemma:relaxation+smoothing}. For any randomized classifier $\bpi$ that is feasible for~\eqref{eq:optimal}, we have \[\mathcal{R}(\bpi_{\blambda^\star}) \leq \mathcal{R}(\bpi)+\frac{\log(|\cA|)}{\beta}\enspace.\]
\end{lemma}
\begin{proof}[\textbf{Proof of Lemma}~\ref{lemma:risk}]
     Fix some randomized prediction function $\bpi$ that is feasible for the problem in~\eqref{eq:optimal}. Then, we can bound its negative risk as follows
    \begin{equation}
    \begin{aligned}
        -\risk(\bpi)
        &= -\Exp_\bX\left[\sum_{a \in [A]}\ell(\bX, a)\pi(a \mid \bX)\right]\\
        &\stackrel{(a)}{\leq}
        \Exp_\bX\left[\sum_{a \in \cA}\left(-\ell(\bX, a) -\scalar{\blambda^{\star}}{\bc(\bX, a)}\right)\pi(a \mid \bX)\right] \\
        &\stackrel{(b)}{\leq} \Exp_\bX\left[\lse\left(\left(-\ell(\bX, a) -\scalar{\blambda^{\star}}{\bc(\bX, a)}\right)_{a \in \cA}\right)\right]\\
        &\stackrel{(c)}{=}
        \Exp_\bX\left[\sum_{a \in \cA} \left(-\ell(\bX, a) -\scalar{\blambda^{\star}}{\bc(\bX, a)}\right)\pi^{\star}(a \mid \bX) - \frac{1}{\beta}\Psi(\bpi^\star(\cdot \mid \bX))\right]\\
        &\stackrel{(d)}{\leq}
        \Exp_\bX\left[\sum_{a \in \cA} \left(-\ell(\bX, a) -\scalar{\blambda^{\star}}{\bc(\bX, a)}\right)\pi^{\star}(a \mid \bX)\right]
         + \frac{\log(A)}{\beta}\\
        &\stackrel{(e)}{=}
        -\risk(\pi^\star) + \frac{\log(A)}{\beta}\,,
    \end{aligned}
    \end{equation}
for (a) we used the assumption that $\bpi$ is feasible for \eqref{eq:optimal} (\ie $\cC_j(\pi) \leq 0$), which due to the fact that $\blambda^\star\geq 0$ implies
    \begin{align*}
        \scalar{\blambda^\star}{-\Exp_\bX\left[\sum_{a \in \cA}\bc(\bX, a) \pi(a \mid \bX)\right]} \geq 0 \enspace,
    \end{align*}
    since every term in the summation is non-negative; (b) uses the fact that $\lse(\bw) \geq \scalar{\bw}{\bp}$ for any probability vector $\bp$ (see Lemma~\ref{lemma:lse_variational} for details); (c) relies on the variational representation of the $\lse$, recalled in Lemma~\ref{lemma:lse_variational} and the definition of $\bpi^\star(\cdot \mid \bX)$; (d) uses the uniform bound on the entropy; (e) the last equality relies on the complementary slackness condition~\eqref{kkt} of Lemma~\ref{lem:KKT}. It ensures that
    \begin{align*}
        \lambda^\star_j \left(-\Exp_\bX\left[\sum_{a \in \cA}c_j(\bX, a) \pi^\star(a \mid \bX)\right] \right) = 0 \qquad \forall j \in [M]
    \end{align*}
     implying that
     \begin{align*}
         -\Exp_\bX\left[\sum_{a \in \cA} \scalar{\blambda^{\star}}{\bc(\bX, a)}\pi^{\star}(a \mid \bX)\right]
         = 0 \enspace.
     \end{align*}
     The proof is concluded.
\end{proof}
The two results above suggest that $\bpi_{\blambda^\star}$ serves as a strong alternative to the actual minimizer of~\eqref{eq:optimal}, provided that $\beta > 0$ remains within a reasonable range. Accordingly, our goal is to develop an algorithm that replicates the performance of \(\bpi_{\blambda^\star}\). To this end, we introduce the following parametric family of prediction functions. For any \(\blambda \geq 0\), we define
\begin{align}
\label{eq:any_entropic}
\pi_{\blambda}(a \mid \bx) \propto \exp\left(\beta\left(
-\bL(\bx) - \bC(\bx)^\top\blambda\right)_a\right) \text{ for } a \in \cA \enspace.
\end{align}
A key property of this family is its parametric nature: when \(\blambda = \blambda^\star\), it yields a randomized classifier that closely approximates a solution to~\eqref{eq:optimal}. This reformulates the problem into finding a good approximation for \(\blambda^\star \geq 0\).

As noted in~\citep{NEURIPS2024_d5c3ecf3}, while $F$ in \eqref{min-lse} is convex, the primary concern is not the functional optimization error but rather approximate stationarity---vanishing gradient map. To address this, we introduce the gradient mapping—a modified version of the gradient that quantifies the distance between a given point and its projection onto the feasible region after a gradient step. Specifically, for $\alpha > 0$, we define
\begin{equation}
\label{eq:grad_map_to_grad_plus}
\begin{aligned}
    \bG_{\alpha}\left(\blambda\right)
    \eqdef \frac{\blambda- \left(\blambda-\alpha\nabla F\left(\blambda\right)\right)_{+}}{\alpha}\enspace.
\end{aligned}
\end{equation}
Finally, we present our main observation.
\begin{lemma}
    \label{lemma:grad_map_stat}
    Let Assumptions~\ref{ass:slater} be satisfied. For any $\alpha > 0, \beta > 0, \blambda \geq 0$, the randomized classifier $\bpi_{\blambda}$ satisfies
\begin{align*}
    \sqrt{\sum_{j \in [M]}\left(\cC_j(\pi_{\blambda})\right)_+^2} \leq \norm{\bG_{\alpha}\left(\blambda\right)}\enspace,
\end{align*}
Furthermore, for any $\bpi$ that is feasible for~\eqref{eq:optimal}
\begin{align*}
    \mathcal{R}(\bpi_{\blambda})
    \leq \risk(\bpi) + \left(\norm{\blambda} + \alpha\Exp_{\bX}\pqnorm{1}{2}{\bC(\bX)}\right) \cdot\norm{\bG_{\alpha}(\blambda)} + \frac{2\log(|\cA|)}{\beta} \enspace.
\end{align*}
\end{lemma}
\begin{proof}[\textbf{Proof of Lemma}~\ref{lemma:grad_map_stat}]
    Fix arbitrary $\blambda \geq 0$ and consider $\bpi_{\blambda}$, defined in~\eqref{eq:any_entropic}. To ease the notation, in this proof, we write ${\bpi}$ instead of $\bpi_{\blambda}$. The proof will consist of two parts, in the first we establish constraint violation control and then provide risk bound.\\
\textbf{Part I.}
We first make the following elementary observation: for any $\alpha, a\geq0$ and $b \in \bbR$, we have
\begin{align*}
    &\left|\frac{a-(a-\alpha b)_{+}}{\alpha}\right| = \left|\frac{a-\max\{0; a-\alpha b\}}{\alpha}\right| = \left|\min\left\{\frac{a}{\alpha};b\right\}\right| \geq \left|\min\left\{0;b\right\}\right| \geq (-b)_{+} \enspace.
\end{align*}
The above in conjunction with the definition of $\bG_\alpha$ in~\eqref{eq:grad_map_to_grad_plus}, yields
\begin{align}
    \label{unfairness-init}
    \norm{\left(-\nabla F(\blambda)\right)_{+}} \leq \norm{\bG_{\alpha}(\blambda)}\qquad \forall \blambda \geq 0\enspace.
\end{align}
Computing the gradient of $F$ and using the expression for $\bpi$ in~\eqref{eq:any_entropic}, we observe that
\begin{equation}
    \label{eq:true_grad_recalled}
\begin{aligned}
    &\nabla_{\lambda_j}F({\blambda}) = -\Exp_\bX\left[\sum_{a \in \cA}C_j(\bX, a) \pi(a \mid \bX)\right] = -\cC_j(\bpi) \enspace.
\end{aligned}
\end{equation}
The above two displays conclude the proof of the first part of the Lemma.

\noindent\textbf{Part II.}
We note that $\bpi_{\blambda}$ is a unique solution to
\begin{align*}
    \min_{\bpi} \class{L}(\bpi, \blambda)\enspace,
\end{align*}
where $\class{L}$ is the Lagrangian defined in~\eqref{eq:lagrangian}. Furthermore, $\min_{\bpi} \class{L}(\bpi, \blambda) = -F(\blambda) = \class{L}(\bpi_{\blambda}, \blambda)$. Hence, recalling \eqref{eq:true_grad_recalled}, we get
\begin{align*}
    \risk_\beta(\bpi_{\blambda}) + F(\blambda)
    &= \risk_\beta(\bpi_{\blambda}) - \class{L}(\bpi_{\blambda}, \blambda) = - \sum_{j \in [M]}\lambda_j\cC_j(\bpi) = \scalar{\blambda}{\nabla F({\blambda})} \enspace.
\end{align*}
Now, let us discuss the possible values of each element of $\bG_\alpha$. For any $j \in [M]$, since $\lambda_J \geq 0$, we have
\begin{align*}
    \bG_{\alpha j}(\blambda) =
    \begin{cases}
        \nicefrac{\lambda_j}{\alpha}  &\text{if } \alpha\partial_j F(\blambda) \geq \lambda_j \\
        \partial_j F(\blambda)  &\text{otherwise }
    \end{cases}\,.
\end{align*}
Let us denote by $\tilde\blambda \eqdef \blambda-\alpha\nabla F(\blambda)$ and notice, that
\begin{align*}
    \lambda_j\partial_j F(\blambda) &= \alpha \partial_j F(\blambda) \bG_{\alpha j} (\blambda)\ind{\tilde \lambda_j \leq 0} + \lambda_j \bG_{\alpha j} (\blambda) \ind{\tilde \lambda_j > 0}\enspace.
\end{align*}
Thus, bounding the indicators by one, we deduce
\begin{align*}
    \scalar{\blambda}{\nabla F(\blambda)}
    &\leq
    \sum_{j \in [M]}\left( \alpha |\partial_j F(\blambda) \bG_{\alpha j} (\blambda)| + |\lambda_j \bG_{\alpha j} (\blambda)|\right) \leq
    \left(\norm{\blambda} + \alpha\norm{\nabla F(\blambda)}\right)\norm{\bG_{\alpha}(\blambda)} \\
     &\leq \left(\norm{\blambda} + \alpha\Exp_\bX \pqnorm{1}{2}{\bC(\bX)}\right)\norm{\bG_{\alpha}(\blambda)}\enspace,
\end{align*}
where the second inequality comes from Cauchy-Schwartz and the last one follows from Lemma~\ref{lemma:normF-bd} that bounds $\norm{\nabla F(\blambda)}$.

To conclude the proof, we observe that $\min_{\blambda} F(\blambda) = -\risk_\beta(\bpi_{\blambda^\star})$, the fact that $ \risk_{\beta}(\bpi) \leq \risk(\bpi) + \tfrac{\log(2L + 1)}{\beta} $ for all $\bpi$, and invoke Lemma~\ref{lemma:risk}.
\end{proof}
Lemma~\ref{lemma:grad_map_stat} demonstrates the fact, that if we are able to control the norm of the gradient mapping, then we can obtain a reliable estimator in terms of risk and "constraint satisfaction". This is the central observation of this work and our theoretical results are all derived from it.

\paragraph{Towards data-driven post-processing.}
Lemma~\ref{lemma:grad_map_stat} is particularly useful when both the loss vector \(\bL\) and the constraint matrix \(\bC\) in~\eqref{min-lse} are known in advance, with the only unknown component in $F$ being the expectation operator \(\Exp_{\bX}\). However, as demonstrated in several examples above, we often do not know \(\bL\), \(\bC\), or both. In these cases, Lemma~\ref{lemma:grad_map_stat} provides limited practical guideline, making it difficult to find \(\blambda\) that minimizes \(\norm{\bG_{\alpha}(\blambda)}\). Nonetheless, when estimators such as \(\hat{\bL}\), \(\hat{\bC}\) of \(\bL\), \(\bC\) are available, we can derive more informative results.

Before stating them, let us introduce plug-in versions of $F, \pi_{\blambda}$, and $\bG_{\alpha}$ as follows:
\begin{align}
    \label{eq:hatF}
    \hat{F}(\blambda) &\eqdef \Exp_\bX\left[\lse\left(
        -\hat{\bL}(\bX) - \hat{\bC}(\bX)^\top\blambda\right)\right]\\
        \label{eq:hatPi}
    \hat{\pi}_{\blambda}(a \mid \bx) &\propto \exp\left(\beta\left(
-\hat{\bL}(\bx) - \hat{\bC}(\bx)^\top\blambda\right)_a\right) \text{ for } a \in \cA\\
\label{eq:HatG}
\hat{\bG}_{\alpha}\left(\blambda\right)
    &\eqdef \frac{\blambda- \left(\blambda-\alpha\nabla \hat{F}\left(\blambda\right)\right)_{+}}{\alpha}\enspace.
\end{align}
We are in position to derive our main master lemma, that will allow us to use any off-the-shelf stochastic optimization algorithm and deduce statistically sound results.
\begin{lemma}
     \label{lemma:hat_grad_map_stat}
    Let \(\hat{\bL}\), \(\hat{\bC}\), be some estimators of \(\bL\), \(\bC\).
    Let Assumptions~\ref{ass:slater} be satisfied. Assume additionally that there exists $\bpi$ such that $\Exp[\hat{\bC}(\bX) \bpi(\bX)] < 0$.
    Let
    \begin{align*}
        \hat{\Delta}_{\loss} \eqdef \Exp\|\bL(\bX) - \hat{\bL}(\bX)\|_{\infty} \quad\text{ and }\quad \hat{\Delta}_{\constr}^2 \eqdef {\Exp\pqnormin{1}{2}{\bC(\bX) - \hat{\bC}(\bX)}^2}\enspace,
    \end{align*}
    stand for estimation error of the loss function and constraints.
    For any $\alpha > 0, \beta > 0, \blambda \geq 0$, the randomized classifier $\hat{\bpi}_{\blambda}$ satisfies
\begin{align*}
    \sqrt{\sum_{j \in [M]}\left(\cC_j(\hat{\bpi}_{\blambda})\right)_+^2} \leq \normin{\hat{\bG}_{\alpha}\left(\blambda\right)} + \hat{\Delta}_{\constr}\enspace,
\end{align*}
Furthermore, for any $\bpi$ that is feasible for~\eqref{eq:optimal}
\begin{align*}
    \mathcal{R}(\hat{\bpi}_{\blambda})
    \leq \risk(\bpi) + \big(\norm{\blambda} + \alpha\,\Exp_{\bX}\pqnormin{1}{2}{\hat{\bC}(\bX)}\big) \cdot\normin{\hat{\bG}_{\alpha}(\blambda)} + 2\hat{\Delta}_{\loss} + \hat{\Delta}_{\constr}\cdot\norm{\blambda} +  \frac{2\log(|\cA|)}{\beta} \enspace.
\end{align*}
\end{lemma}
\begin{proof}[Proof of Lemma~\ref{lemma:hat_grad_map_stat}]
Fix arbitrary $\blambda \geq 0$ and consider $\hat{\bpi}_{\blambda}$, defined in~\eqref{eq:hatPi}. To ease the notation, in this proof, we write ${\hat{\pi}}$ instead of $\hat{\bpi}_{\blambda}$.\\
\textbf{Part I.}
Thanks to the triangle's inequality
\begin{align*}
    \sqrt{\sum_{j \in [M]}\left({\cC}_j(\hat{\pi})\right)_+^2} \leq \sqrt{\sum_{j \in [M]}\big(\hat{\cC}_j(\hat{\pi})\big)_+^2} + \sqrt{\sum_{j \in [M]}\big|\hat{\cC}_j(\hat{\pi}) - {\cC}_j(\hat{\pi})\big|^2}\enspace.
\end{align*}
To bound the first term on the right hand side of the above inequality, we use similar arguments as in the first part of the proof of Lemma~\ref{lemma:grad_map_stat} and deduce that
\begin{align*}
    {\sum_{j \in [M]}\big(\hat{\cC}_j(\hat{\pi})\big)_+^2} \leq\normin{\hat{\bG}_{\alpha}(\blambda)}^2\enspace.
\end{align*}
It remains to bound the second term. To that end, we proceed as
\begin{align*}
    \sum_{j \in [M]}\big|\hat{\cC}_j(\hat{\pi}) - {\cC}_j(\hat{\pi})\big|^2
    &\leq
    \Exp\sum_{j \in [M]}  
    \left|\sum_{a \in \cA}\hat{\pi}(a \mid \bX)(C_j(\bX, a) - {\hat{C}_j}(\bX, a))
    \right|^2\\
    &=
    \Exp\|(\bC(\bX) - {\hat{\bC}}(\bX))\hat{\bpi}(\bX)\|^2
    \leq
    \Exp\pqnormin{1}{2}{\bC(\bX) - {\hat{\bC}}(\bX)}^2\enspace,
\end{align*}
where the first inequality comes from Jensen and the last one from~\citep[Theorem 2.1]{lewis2023}.

\textbf{Part II.} Similarly to the second part of the proof of Lemma~\ref{lemma:grad_map_stat} and thanks to the Slatter's condition, assumed in the statement, we deduce that
\begin{align*}
\hat{\risk}_\beta(\hat{\pi}) + \hat{F}(\blambda) \leq \left(\norm{\blambda} + \alpha\Exp_\bX \pqnormin{1}{2}{\hat{\bC}(\bX)}\right)\normin{\hat{\bG}_{\alpha}(\blambda)}\enspace,
\end{align*}
where we define
\begin{align*}
    \hat{\risk}_\beta(\hat{\pi}) = {\risk}_\beta(\hat{\pi}) + \Exp\sum_{a \in \cA}(\ell(\bX, a) - \hat{\ell}(\bX, a))\hat{\pi}(a \mid \bX)\enspace.
\end{align*}
Due to the definition of $F$ and $\hat{F}$ and the fact that $\|\nabla\lse(\cdot)\|_1 = 1$ (see Lemma~\ref{lemma:lse}), we have that
\begin{align*}
    |F(\blambda) - \hat{F}(\blambda)|
    &\leq
    \Exp\|\bL(\bX) + \bC(\bX)^\top\blambda - \hat{\bL}(\bX) - \hat{\bC}(\bX)^\top\blambda\|_{\infty}\\
    &\leq \Exp\|\bL(\bX) - \hat{\bL}(\bX)\|_{\infty} + \Exp\|(\bC(\bX) - \hat{\bC}(\bX))^\top\blambda\|_{\infty}\enspace.
\end{align*}
The combination of the above two displays yield
\begin{align*}
{\risk}_\beta(\hat{\pi}) + {F}(\blambda) \leq \big(\norm{\blambda} &+ \alpha\Exp_\bX \pqnormin{1}{2}{\hat{\bC}(\bX)}\big)\normin{\hat{\bG}_{\alpha}(\blambda)}\\
&+ 2\Exp\|\bL(\bX) - \hat{\bL}(\bX)\|_{\infty} + \|\blambda\|\Exp\pqnormin{2}{\infty}{(\bC(\bX) - \hat{\bC}(\bX))^\top}\enspace.
\end{align*}
Noticing that~\cite[See Theorem 2.1 (ii) and (vi)]{lewis2023}
\begin{align*}
    \pqnormin{2}{\infty}{(\bC(\bX) - \hat{\bC}(\bX))^\top} = \pqnormin{1}{2}{\bC(\bX) - \hat{\bC}(\bX)}\enspace,
\end{align*}
we conclude with identical arguments as for Lemma~\ref{lemma:grad_map_stat}.
\end{proof}

\subsection{Discussion on Lemma~\ref{lemma:hat_grad_map_stat}.}
\label{subsec:discussion_on_methodo}
Lemma~\ref{lemma:hat_grad_map_stat} is designed for practical applications. It states that to “post-process” given estimators of the loss and constraint functions, it suffices to find \(\blambda \geq 0\) such that \(\normin{\hat{\bG}_{\alpha}(\blambda)}\) is as small as possible. The key advantage of Lemma~\ref{lemma:hat_grad_map_stat} over its oracle counterpart, Lemma~\ref{lemma:grad_map_stat}, is that the only unknown quantity in the definition of $\hat{F}$ is the expectation with respect to the marginal distribution of \(\bX\).

Notably, for \(\bX \sim \Prob_{\bX}\), the mapping
\begin{align*}
\blambda \mapsto \lse\left(
-\hat{\bL}(\bX) - \hat{\bC}(\bX)^\top\blambda\right)\enspace,
\end{align*}
is an unbiased estimator of $\hat{F}$, while its gradient serves as an unbiased estimator of $\nabla F$. This property is particularly useful, as it enables the selection of an appropriate \(\blambda \geq 0\) using stochastic optimization methods. In other words, given $T$ i.i.d. samples from $\Prob_{\bX}$ that could potentially arrive in a stream, our goal is to build $\hat{\blambda}_T$ such that
\begin{align*}
    \Expf\|\hat{\bG}_{\alpha}(\hat{\blambda}_T)\|^2 \quad \text{is as small as possible},
\end{align*}
as a function of $T$. Then, plug-in $\hat{\blambda}_T$ into $\hat{\bpi}_{\blambda}$ and invoking Lemma~\ref{lemma:hat_grad_map_stat}, we would obtain a post-processing guarantee for any base estimator of the loss and constraints.

More qualitatively, we note that the error in loss estimation is measured using the $\infty$-norm, which is advantageous as it can scale with \(\log(|\cA|)\) rather than exhibiting a polynomial dependence on the number of classes. Similarly, the error in constraint estimation is measured using the \(\pqnormin{1}{2}{\cdot}\) norm, leading to a more favorable dependence on the size of \(\cA\).

A key feature of the above results is that when the constraint matrix is known—such as in classification with rejection under a controlled rejection rate or in classification under the demographic parity constraint in the awareness setting—the first bound in Lemma~\ref{lemma:hat_grad_map_stat} is essentially identical to that of its oracle counterpart, Lemma~\ref{lemma:grad_map_stat}. This aligns with the literature on post-processing algorithms and has been observed in various settings~\citep{schreuder2021,denis2021,Denis2017}, among others.

\subsection{Examples of estimators for loss and constraints under various assumptions.} As illustrated by the examples in Section~\ref{sub:examples_of_problems}, both the loss and constraint functions often depend on the conditional distribution of $Y \mid \bX$ or $S \mid \bX$, particularly in the fairness settings we have considered. In this section, we review some classical results on estimating conditional distributions in the discrete case and examine their implications for the errors $\hat{\Delta}_{\loss}$ and $\hat{\Delta}_{\constr}$ introduced in Lemma~\ref{lemma:hat_grad_map_stat}.

To focus on the main ideas without delving into technicalities, we assume that the target quantity is $\eta(\bx) = \Prob(Y = 1 \mid \bX = \bx)$. The multi-class case can be treated analogously via a one-vs-all approach, by constructing a separate estimator for each class. Specifically, when $Y \in [K]$, one can define $K$ functions $p_y(\bx) = \Prob(Y = y \mid \bX = \bx)$, each estimated from a binary dataset of the form $(\bX_i, \ind{Y_i = y})$ for every $y \in [K]$. A similar strategy applies when addressing the fairness examples in Section~\ref{sub:examples_of_problems}, where one constructs functions $\tau_s(\bx) = \Prob(S = s \mid \bX = \bx)$, requiring the resolution of $|\cS|$ binary classification problems. Hence, restricting our discussion to the estimation of $\eta(\bx) = \Prob(Y = 1 \mid \bX = \bx)$ is not limiting, as more general cases follow similar lines.
There are several settings in which optimal finite-sample estimation bounds for $\eta$ are known. Below, we provide some details on the estimation procedure under nonparametric assumptions. For parametric examples, we refer the reader to~\cite{van2008high}.

Assume that $\cX = [-1, 1]^d$, for any multi-index $\balpha = (\alpha_1, \ldots, \alpha_d) \in \bbN^d$ and $\bx \in \cX$, define $\bx^{\balpha} \eqdef x_1^{\alpha_1} \cdot \ldots \cdot x_d^{\alpha_d}$ and $|\balpha| = \sum_{j = 1}^d \alpha_j$. Let $K$ be a kernel and $h > 0$ a bandwidth, the local polynomial estimator~\citep[see e.g.,][]{tsybakov2009nonparametric} is defined via a polynomial $\hat{\theta}_{\bx}$ of degree $\ell \geq 0$ that minimizes
\begin{align*}
    \sum_{i = 1}^n\big(Y_i - \hat{\theta}_{\bx}(\bX_i - \bx)\big)^2K\left(\frac{\bX_i - \bx}{h}\right)\,.
\end{align*}
The local polynomial estimator $\hat{\eta}(\bx)$ of degree $\ell \geq 0$ of the value $\eta(\bx)$ is then defined as $\hat{\eta}(\bx) \eqdef \hat{\theta}_{\bx}(0)$ if the minimizer of the above problem is unique and is set to zero otherwise.
To exhibit the statistical result available for this estimator, let us first introduce some key assumptions.
\begin{assumption}
    \label{ass:non-para}
    We assume that there exist $\beta \in \bbN \setminus \{0\}$ and $L > 0$, such that
    \begin{align*}
        |\eta(\bx) - \eta_{\bx}(\bx')| \leq L \|\bx - \bx'\|^\beta\,,
    \end{align*}
    where $\eta_{\bx}$ is Taylor polynomial of degree $\beta$ at $\bx$,
    \begin{align*}
        \eta_{\bx}(\bx') = \sum_{|\alpha| \leq \beta}\frac{(\bx - \bx')^\alpha}{\balpha!} D^\balpha \eta(\bx)\,,
    \end{align*}
    where $\balpha ! = \alpha_1 ! \cdot \ldots \cdot \alpha_d!$ and $D^\balpha \eqdef \tfrac{\partial^{|\balpha|}}{\partial x_1^{\alpha_1} \cdots \partial x_d^{\alpha_d}}$\,. We additionally assume that the support of $\Prob_{\bX}$ is $\cX = [-1, 1]^d$ and it admits density uniformly lower and upper bounded on $\cX$.
\end{assumption}
The next result is borrowed from~\cite{tsybakov2007fast}, who considered local polynomial estimator for building plug-in classification rules.
\begin{theorem}
    Let Assumption~\ref{ass:non-para} be satisfied, then there exists a choice of kernel $K$ and bandwidth $h = h_n$, such that the locally polynomial estimator satisfies
    \begin{align*}
        \Expf\|\hat{\eta} - \eta\|_1 \leq C n^{-\frac{\beta}{2\beta + d}}\,,
    \end{align*}
    for some $C > 0$ that is independent of $n$.
\end{theorem}
The conclusion from the above result, in our context, is that under nonparametric assumptions, one can guarantee that both $\hat{\Delta}_{\loss}$ and $\hat{\Delta}_{\constr}$ scale as $n^{-\frac{\beta}{2\beta + d}}$, where $n$ denotes the size of the labeled dataset used to construct the initial estimators.

\section{Proposed algorithm}
\label{sec:algorithm}
\begin{algorithm}[t]
\caption{$\texttt{COPT}(T, \beta, K, \class{A}, \ell, \bc, \bB)$}\label{alg:main}
\begin{algorithmic}[1]
    \STATE {\bfseries Input:} number of stochastic gradient evaluations $T \geq 1$; regularization $\beta > 0$; number of classes $K \geq 2$; set of possible predictions $\class{A}$; loss estimator $\hat{\ell}(\cdot, \cdot)$; constraint estimators $\hat{C_j}(\cdot, \cdot)$.
   \STATE { Set} parameters: $\sigma^2 \geq \Exp_\bX \pqnormin{1}{2}{\hat{\bC}(\bX)}$, $L=2\beta\sigma^2$\;
    \STATE { Set} $\blambda \mapsto \hat{F}(\blambda)$ as defined in Eq.~\eqref{eq:hatF}\;
   \STATE Run a black-box optimizer $\class{O}(\hat{F}, \sigma^2, L, T)$ on function $\hat{F}$ having access to $T$ stochastic gradient evaluations (see~\eqref{eq:stoch_grad}) with variance $\sigma^2$ and smoothness parameter $L$ to obtain $\hat\blambda_T$\;
   \RETURN $\hat{\pi}_{\hat\blambda_T}(\cdot \mid \cdot)$ as defined in~\eqref{eq:hatPi}\;
\end{algorithmic}
\end{algorithm}
In this section we present details on the family of proposed algorithms. Namely, given any stochastic first order optimization algorithm, we turn it into a post-processing algorithm for the considered constraint multi-class classification problem.
As it was mentioned in Section~\ref{subsec:discussion_on_methodo}, our goal is to instantiate a stochastic optimization algorithm that is tailored for minimization of the norm of the gradient-mapping.

First of all, we need to provide the explicit expression of the gradient of $\hat{F}$.
\begin{align}
    \label{eq:true_grad}
    \nabla\hat{F}(\blambda) = - \Exp_\bX\left[\hat{\bC}(\bX)\bsigma\left(\beta\left(
    -\hat{\bL}(\bX) - \hat{\bC}(\bX)^\top\blambda\right)\right)\right] \enspace.
\end{align}
A stochastic gradient for $\hat{F}$ can be obtained by sampling independently from the past a new point $\bX \sim \Prob_{\bX}$. Namely, for a sample $\bX \sim \Prob_\bX$ and a point $\blambda = (\lambda_j)_{j\in[M]}$, we set
\begin{align}
    \label{eq:stoch_grad}
    \bg(\blambda; \bX) = -\hat{\bC}(\bX)\bsigma\left(\beta\left(
    -\hat{\bL}(\bX) - \hat{\bC}(\bX)^\top\blambda\right)\right) \enspace.
\end{align}
The following relation is then ensured
\begin{align*}
    \nabla\hat{F}(\blambda) = \Exp_{\bX \sim \Prob_{\bX}}[\bg(\blambda; \bX)]\enspace,
\end{align*}
implying that $\bg(\blambda; \bX)$ is an unbiased estimator of $\nabla\hat{F}(\blambda)$. Secondly, as it is common for many (stochastic) first order optimization algorithms, we need to establish a control of the variance of $\bg(\cdot; \bX)$ and the regularity properties of $\hat{F}$.
We state the following two results that accomplish this goal and will allow us to instantiate some black-box stochastic optimization algorithm.
\begin{lemma}[{Variance}]
    \label{lemma:sigma-bd}
    It holds that $\Exp_\bX \normin{g(\blambda; \bX) - \nabla \hat{F} (\blambda)}^2 \leq \Exp_\bX \pqnormin{1}{2}{\hat{\bC}(\bX)}^2$.
\end{lemma}
\begin{proof}[\textbf{Proof of Lemma}~\ref{lemma:sigma-bd}]
From the expressions of $g(\cdot; \bX)$ and $\nabla \hat{F}(\cdot)$ and using the previous relation $\Exp_{\bX}[g(\cdot; \bX)] = \nabla \hat{F}(\cdot)$, we deduce
\begin{align*}
    &\Exp_\bX \normin{g({\blambda} ;\bX) - \nabla \hat{F} (\blambda)}^2
    \leq\Exp_\bX \normin{\hat{\bC}(\bX)\bsigma\big(\beta\big(-\hat{\bL}(\bX) - \hat{\bC}(\bX)^\top\blambda\big)\big)}^2 \leq \Exp_\bX \pqnormin{1}{2}{\hat{\bC}(\bX)}^2 \enspace,
\end{align*}
where the last inequality follows from the definition of $\pqnorm{1}{2}{\cdot}$ and the fact that $\norm{\bsigma(\cdot)}_1=1$.
The proof is concluded.
\end{proof}

\begin{lemma}[{Regularity of $\hat{F}$}]
    \label{lemma:L-bd}
    The function $\blambda \mapsto \hat{F}(\blambda)$ is convex and its gradient is $(2\beta\Exp_\bX \pqnormin{1}{2}{\hat{\bC}(\bX)}^2)$-Lipschitz.
\end{lemma}
\begin{proof}[\textbf{Proof of Lemma}~\ref{lemma:L-bd}]
First, let us recall the expression of $\hat{F}(\blambda)$, given as
\begin{align*}
    \hat{F}(\blambda) = \Exp_\bX\left[\lse\left(
    -\hat{\bL}(\bX) - \hat{\bC}(\bX)^\top\blambda\right)\right] \enspace.
\end{align*}
To show convexity, we first note that $\lse(\cdot)$ is convex. Then, since convexity is preserved under affine transformations, then the mapping $\blambda \mapsto \lse\left(
    -\hat{\bL}(\bx) - \hat{\bC}(\bx)^\top\blambda\right)$ is convex for all $\bx$. Thus, its expectation is convex as well, showing that $\hat{F}$ is convex.

Using the chain rule, we derive the expression for the Hessian of $\hat{F}$, given as
\begin{align*}
    \nabla^2 \hat{F}(\blambda) &= \Exp_\bX\left[(-\hat{\bC}(\bX)) \nabla^2 \lse(-\hat{\bC}(\bX)^\top\blambda - \hat{\bL}(\bX))(-\hat{\bC}(\bX)^\top)\right]\enspace.
\end{align*}
To show the Lipschitzness, it is sufficient to provide an upper bound on the operator norm of $\nabla^2 F(\blambda)$. From Lemma~\ref{lemma:lse} we have that $\nabla^2 \lse(\ba) = \beta(\diag(\bsigma(\beta\ba))-\bsigma(\beta\ba)\bsigma(\beta\ba)^\top)$ for any $\ba\in\bbR^{|\cA|}$. Combining this with Lemma~\ref{lemma:normHessian} and Lemma~\ref{lemma:norm-12} provided in appendix, we obtain
\begin{align*}
    \opnormin{\nabla^2 \hat{F}(\blambda)} \leq 2\beta \Exp_\bX \pqnormin{1}{2}{\hat{\bC}(\bX)}^2\enspace.
\end{align*}
The proof is concluded.
\end{proof}
The previous two lemmas are derived for optimization purposes rather than statistical ones. Explicit bounds on the variance of the stochastic gradient and conditions on the regularity of the objective function are crucial for establishing the convergence rate of an optimization algorithm. These results enable the use of any stochastic optimization method designed to minimize the norm of the gradient mapping. For this reason, we present them as standalone results, allowing for potential use in future research.

That being said, to make our analysis complete, the next section introduces a concrete optimization algorithm. This algorithm was originally developed by~\cite{allenzhu2021make}, later refined by~\cite{foster2019complexity}, and further adapted to meet the specific needs of this work by~\cite{NEURIPS2024_d5c3ecf3}.

\subsection{Black-box stochastic optimization reduction}
In this section, we show how black-box stochastic optimization guarantees translate to statistical guarantees for our problem. We start with establishing the notion of a black-box algorithm.
Let $f : \bbR^M_+ \times \cX \to \bbR$ and
$P$ be some probability distribution on $\cX$; define
\begin{align}
    \label{eq:optim_general}
    F(\blambda) \eqdef
    \Exp_{P}\left[ f(\blambda, \bX)\right]\enspace.
\end{align}
We consider the following class of functions and distributions $P$:
\begin{definition}
    \label{def:optim_class}
    For some $L > 0$, $\sigma > 0$, let $\cF_{L, \sigma_2}$ stand for a class of functions $F$ written in the form~\eqref{eq:optim_general}, such that
    \begin{enumerate}
        \item $\blambda \mapsto f(\blambda, \bx)$ is convex for every $\bx \in \cX$;
        \item $\|\nabla F(\blambda) - \nabla F(\blambda')\| \leq L \|\blambda - \blambda'\| $ for $\blambda, \blambda' \in \bbR^M_+$;
        \item $\displaystyle\sup_{\blambda \in \bbR^M_+} \Exp_{P}
        [\|\nabla_{\blambda}f(\blambda, \bX) - \nabla F(\blambda)\|^2] \leq \sigma^2$.
    \end{enumerate}
\end{definition}
First of all, observe that thanks to Lemma~\ref{lemma:sigma-bd} and~\ref{lemma:L-bd}, the function $\blambda \mapsto \hat{F}(\blambda)$, defined in~\eqref{eq:hatF}, belongs to $\cF_{L, \sigma_2}$, with
\begin{align*}
    f(\blambda, \bx) = \lse\big(
        -\hat{\bL}(\bx) - \hat{\bC}(\bx)^\top\blambda\big)\, , \quad
        \sigma^2 = \Exp_\bX \pqnormin{1}{2}{\hat{\bC}(\bX)}^2 \, ,  \text{ and}  \quad
    L = 2\beta \sigma^2\,.
\end{align*}
Our goal is to take an arbitrary stochastic black-box algorithm that we define below and comment on the way it can be simulated.
\begin{definition}[A black box optimizer]
    \label{def:black_box}
   For $L, \sigma^2 > 0$, $F \in \cF_{L, \sigma_2}$ and $T \geq 1$, a black box optimizer $\cO(F, \sigma^2, L, T)$ is any procedure that, on a basis of $\bX_1, \ldots, \bX_T$ \iid samples from $P$,
   returns $\hat{\blambda}_T$ satisfying
   \begin{align*}
        \mathbf{E}\bigg\|\frac{\hat{\blambda}_T- \big(\hat{\blambda}_T-\alpha\nabla {F}\big(\hat{\blambda}_T\big)\big)_{+}}{\alpha}\bigg\|^2 \leq \Psi^2_{\alpha}(T, \sigma^2, L)\,\,,
    \end{align*}
    for some $\alpha > 0$ where $\Psi^2_{\alpha}(T, \sigma^2, L)$ is some black-box quantity that controls the norm of the gradient mapping.
\end{definition}

The following corollary is an immediate consequence of Lemma~\ref{lemma:hat_grad_map_stat}, Cauchy-Schwartz inequality and the mere definition of the black-box otpimizer above.
\begin{corollary}
    \label{cor:main}
    Algorithm~\ref{alg:main}, with an arbitrary black-box optimizer $\cO$ from Definition~\ref{def:black_box} with $\sigma^2 = \Exp_{\bX}\pqnormin{1}{2}{\hat{\bC}(\bX)}^2$ and $L = 2\beta\sigma^2$, satisfies
    \begin{align*}
    \Expf\sqrt{\sum_{j \in [M]}\left(\cC_j(\hat{\pi}_{\hat{\blambda}_T})\right)_+^2} \leq \Psi_{\alpha}(T, \sigma^2, L) + \hat{\Delta}_{\constr}\enspace.
\end{align*}
Furthermore, for any $\pi$ that is feasible for~\eqref{eq:optimal}
\begin{align*}
    \Expf[\mathcal{R}(\hat{\pi}_{\hat{\blambda}_T})]
    \leq \risk(\pi) + &\sqrt{\Expf\big(\|\hat{\blambda}\| + \alpha\sigma\big)^2} \Psi_{\alpha}(T, \sigma^2, L)\\
    &+ 2\Expf[\hat{\Delta}_{\loss}] + \Expf[\hat{\Delta}_{\constr}\normin{\hat{\blambda}_T}] +  \frac{2\log(|\cA|)}{\beta} \enspace.
\end{align*}
\end{corollary}

\paragraph{Concrete example of an algorithm and explicit post-processing rates.}
To instantiate Corollary~\ref{cor:main}, we require an example of a stochastic optimization algorithm for which the quantity $\Psi_{\alpha}(T, \sigma^2, L)$ can be made sufficiently small. This challenge, in the context of convex problems, has been previously addressed by~\cite{allenzhu2021make} and~\cite{foster2019complexity}, who study the \texttt{SGD3} algorithm and its variant tailored specifically for minimizing the norm of the gradient of a convex function.
Below, we present a version of a result from~\citet[Theorem C.3]{NEURIPS2024_d5c3ecf3}, who, by combining insights from~\cite{allenzhu2021make} and~\cite{foster2019complexity}, provided an adaptations of the proof for the \texttt{SGD3} algorithm that are well suited to our setting.
\begin{theorem}
\label{thm:grad-sq:AC-SA}
Let Assumption~\ref{def:optim_class} be satisfied and set $\blambda^{\star} \in \argmin_{\blambda\in \bbR^M_+}F(\blambda)$ and $\blambda_0 \in \bbR^M_+$ be an initialization. Let also $\mu\in(0,L]$ and $T>4\sqrt{\frac{L}{\mu}}\floor*{\log_2\frac{L}{\mu}}$. Then for $\alpha=\tfrac{1}{2^{J+2}\mu}$, with $J=\floor*{\log_2\frac{L}{\mu}}$, $\texttt{SGD3}(F,\blambda_0,\mu,L,T)$ with $\texttt{AC-SA}$ outputs $\hat\blambda$ that satisfies, for some absolute constant $C > 0$
\begin{align*}
    \Psi_{\alpha}^2(T, \sigma^2, L)
    \leq C\left\{\left(\frac{L \mu}{T^2}\log_2^3\frac{L}{\mu} + \mu^2\right)\norm{\blambda_0-\blambda^*}^2 + \frac{\sigma^2}{T}\log_2^3\frac{L}{\mu}\right\} \enspace.
\end{align*}
\end{theorem}
It remains to pick a parameter $\mu > 0$ and present the final form of the bound for $F = \hat{F}$ as defined in~\eqref{eq:hatF}.
According to Lemma~\ref{lemma:sigma-bd} and Lemma~\ref{lemma:L-bd} we have that $\sigma^2 = \Exp_\bX \pqnormin{1}{2}{\hat{\bC}(\bX)}^2$ and $L=2\beta\sigma^2$.
Thus, setting \[\beta=\frac{T}{8\log_2 T}\,\text{ and }\, \mu= \frac{2 \sigma^2}{\beta} \,\text{ ensures that }\, \mu \leq L \,\text{ and that }\, T > 4\sqrt{\tfrac{L}{\mu}}\floor*{\log_2\tfrac{L}{\mu}} \text{ for } T\geq2\,.\] We conclude that for $T$ larger than some constant, the conditions of Theorem~\ref{thm:grad-sq:AC-SA} are satisfied for the function $\hat{F}$ and we deduce
\begin{align}
    \label{eq:grad_map_stat_proof}
    \Psi_{\alpha}^2(T, \sigma^2, L) \leq \mathcal{\tilde{O}}\left(\frac{\sigma^2}{T}\left( 1 + \frac{\sigma^2}{T} \normin{\tilde{\blambda}}^2\right)\right) \,\text{ where }\, \tilde{\blambda} \in \argmin_{\blambda \geq 0} \hat{F}(\blambda)\enspace.
\end{align}
Combining the above with Corollary~\ref{cor:main}, allows us to conclude that, given $T$ unlabeled samples from $\Prob_\bX$, \texttt{COPT} algorithm~\ref{alg:main} with \texttt{SGD3} algorithm~\ref{alg:SGD3-ref_main} and $\beta = \widetilde{\cO}(T)$ enjoys
\begin{align*}
    \Expf\sqrt{\sum_{j \in [M]}\left(\cC_j(\hat{\pi}_{\hat{\blambda}_T})\right)_+^2} \leq \hat{\Delta}_{\constr} + \mathcal{\tilde{O}}\left(\frac{\sigma}{\sqrt{T}}\left( 1 + \frac{\sigma}{\sqrt{T}} \normin{\tilde{\blambda}}\right)\right)\enspace.
\end{align*}
Furthermore, for any $\pi$ that is feasible for~\eqref{eq:optimal}
\begin{align*}
    \Expf[\mathcal{R}(\hat{\pi}_{\hat{\blambda}_T})]
    \leq \risk(\pi) + &\sqrt{\Expf\big(\normin{\hat{\blambda}_T} + \alpha\sigma\big)^2} \cdot \mathcal{\tilde{O}}\left(\frac{\sigma}{\sqrt{T}}\left( 1 + \frac{\sigma}{\sqrt{T}} \normin{\tilde{\blambda}}\right)\right)\\
    &+ 2\Expf[\hat{\Delta}_{\loss}] + \Expf[\hat{\Delta}_{\constr}\normin{\hat{\blambda}_T}] +  \mathcal{\tilde{O}}\left(\frac{2\log(|\cA|)}{T}\right) \enspace.
\end{align*}
One of the key attractive features of this bound is the high value of $\beta$, which leads to randomized yet highly concentrated classifiers that retain finite-sample guarantees. For this reason, we believe that the introduced randomization is not particularly restrictive in practice.
We also emphasize that the bound holds under virtually no assumptions. Naturally, the quality of $\hat{\Delta}_{\loss}$ and $\hat{\Delta}_{\constr}$ will depend on the regularity of the underlying distribution, but this concern lies outside the scope of the post-processing procedure itself.

\begin{algorithm}[t]
\caption{$\texttt{AC-SA}(F,\blambda_0,\mu, L, T)$}\label{alg:ACSA_main}
\begin{algorithmic}[1]
   \STATE {\bfseries Input:} function $F$; initial vector $\blambda_0$; parameters $\mu, L \geq 0$; number of iterations $T\geq1$\\
   \STATE $\blambda_0^{\textrm{ag}}=\blambda_0$
   \FOR{$t=1$ {\bfseries to} $T$}
   \STATE sample new $\bX \sim \Prob_{\bX}$, independently from the past
   \STATE $\alpha_t \leftarrow \frac{2}{t+1}$
   \STATE $\gamma_t \leftarrow \frac{4L}{t(t+1)}$
   \STATE $\blambda_t^{\textrm{md}} \leftarrow \frac{(1-\alpha_t)(\mu+\gamma_t)}{\gamma_t+(1-\alpha_t^2)\mu}\blambda_{t-1}^{\textrm{ag}} + \frac{\alpha_t((1-\alpha_t)\mu+\gamma_t)}{\gamma_t+(1-\alpha_t^2)\mu}\blambda_{t-1}$
    \STATE $\blambda_{t} \leftarrow \left(\frac{(1-\alpha_t)\mu+\gamma_t}{\mu+\gamma_t}\blambda_{t-1} + \frac{\alpha_t\mu}{\mu+\gamma_t} \blambda_t^{\textrm{md}}-\frac{\alpha_t}{\mu+\gamma_t}\nabla_{\blambda} f(\blambda_t^{\textrm{md}}, \bX)\right)_+$
   \STATE $\blambda_t^{\textrm{ag}} \leftarrow \alpha_t \blambda_t + (1-\alpha_t)\blambda_{t-1}^{\textrm{ag}}$
   \ENDFOR
   \RETURN $\blambda_t^{\textrm{ag}}$
\end{algorithmic}
\end{algorithm}

\begin{algorithm}[t]
\caption{$\texttt{SGD3}(F,\blambda_0,\mu,L,T)$}\label{alg:SGD3-ref_main}
\begin{algorithmic}[1]
   \STATE {\bfseries Input:} function $F$; initial vector $\blambda_0$; parameters $0<\mu\leq L$; number of iterations $T\geq\Omega\left(\frac{L}{\mu}\log_2\frac{L}{\mu}\right)$
   \STATE $F^{(0)}(\blambda) \leftarrow F(\blambda) + \frac{\mu}{2}\norm{\blambda-\blambda_0}^2; \hat \blambda_0 \leftarrow \blambda_0; \mu_0 \leftarrow \mu$
   \FOR{$j=1$ {\bfseries to} $J=\floor*{\log\frac{L}{\mu}}$}
   \STATE $\hat \blambda_j \leftarrow \texttt{AC-SA}(F^{(j-1)},\hat \blambda_{j-1}, \mu_{j-1}, 2(L+\mu), \frac{T}{J})$
   \STATE $\mu_j \leftarrow 2\mu_{j-1}$
   \STATE $F^{(j)}(\blambda) \eqdef F^{(j-1)}(\blambda) + \frac{\mu_j}{2}\normin{\blambda-\hat \blambda_j}^2$
   \ENDFOR
   \RETURN $\hat \blambda_J$
\end{algorithmic}
\end{algorithm}

\subsection{Related works}
\label{sub:related_works}
There are several existing works that offer a flexible algorithmic framework for post-processing a given classifier to satisfy some kind of additional constraints. Most notably, that of~\cite{celis2019classification} and~\cite{narasimhan2024consistent,narasimhan18a} have similar goals to our, but rely on different tools and provide different theoretical results that we discuss below.

\paragraph{Comparison to~\cite{celis2019classification}.} Theorem 3.2 in~\cite{celis2019classification} bears resemblance to our Lemma~\ref{lem:NP}, as it characterizes the form of the optimal classifier under certain fairness constraints. Interestingly, their result yields a non-randomized classifier without requiring any additional assumptions. This can be theoretically problematic: applying their framework to derive the classical Neyman–Pearson lemma would suggest that the optimal test is never randomized, which is generally false in the case of discrete distributions. Nevertheless, we believe this technical issue has limited practical implications.

From a methodological standpoint,~\cite{celis2019classification} advocate solving the optimization problem that defines $\blambda^\star$ (as in our Lemma~\ref{lem:NP}) and then applying the plug-in principle. This strategy is shared by several works, including~\citep{chzhen2020-discr,lei2014,zeng2022fair,hou2024finite,lei2016}. Notably,~\cite{celis2019classification} assume full access to the underlying distribution and do not provide finite-sample statistical guarantees. When such guarantees are derived in related literature, they typically rely on some form of continuity assumption, which enables bypassing randomized classifiers and developing a consistent theory.

Moreover, many of these methods assume access to an exact solution of the underlying convex optimization problem, whereas approximate solutions may be difficult to handle within their frameworks. In contrast, by explicitly allowing for randomized classifiers, we provide a fully data-driven and distribution-free approach, where each algorithmic step can be implemented in practice.

\paragraph{Comparison to~\cite{narasimhan2024consistent}.} They consider constrained multi-class classification problems where both the risk and constraints are formulated as functions of the confusion matrix, covering most of the examples we study. Their approach is based on plug-in algorithms that leverage a linear minimization oracle (e.g., cost-sensitive classification) and comes with optimization convergence guarantees for several procedures, including Frank-Wolfe, Gradient Descent-Ascent, Ellipsoid, and Bisection methods. However, their analysis focuses primarily on the optimization error of the resulting classifier and does not support stochastic variants of their algorithms. In contrast, we provide end-to-end convergence rates for our method, while preserving the same level of generality in the framework.

Methodologically, their approach relies heavily on the primal-dual structure of the problem, whereas our algorithm operates entirely in the dual domain, invoking a primal-dual link only at the final stage of the procedure. The trade-off for our approach is the randomized nature of the resulting classifier. We also note that both our method and that of~\cite{narasimhan2024consistent} produce randomized classifiers as outputs. However, in our case, the degree of randomization is explicitly controlled by the parameter $\beta > 0$, and for large values of $\beta$, the randomization becomes negligible. This level of control is not necessarily available in the approach proposed by~\cite{narasimhan2024consistent}.

Overall, we believe that our approach is complementary to those of~\cite{celis2019classification} and~\cite{narasimhan2024consistent}, with its main distinguishing advantage being that it is end-to-end theory-driven.

\section{Flexibility of the approach: extensions for set-valued classification}
Let us also discuss an interesting extension of our approach to set-valued classification.
In the standard formulation of set-valued classification, the goal is to construct a set-valued classifier $\Gamma : \cX \to 2^{[K]}$ that, for a given input $\bx \in \cX$, returns a subset $\Gamma(\bx) \subset [K]$ of potentially relevant classes. More formally, given a joint distribution $(\bX, Y) \sim \Prob$ over $\cX \times [K]$, two key performance metrics are typically associated with a given $\Gamma$: the {risk} and the {size}, defined as
\begin{align*}
    \cR(\Gamma) \eqdef \Prob(Y \notin \Gamma(\bX))\quad\text{and}\quad
    \cI(\Gamma) \eqdef \Exp |\Gamma(\bX)|\,,
\end{align*}
where $|\Gamma(\bx)|$ denotes the cardinality of the prediction set at input $\bx$.

Two main optimization paradigms have emerged in the literature: one focuses on {minimizing the risk under a constraint on the average size}~\citep{denis2020consistency}, while the other emphasizes {minimizing the average size under a constraint on the risk}~\citep{sadinle2019least}. These lead to the following two canonical optimization problems:
\begin{align}
    \label{eq:set_valued1}
    \min\{\cR(\Gamma)\,:\, \cI(\Gamma) \leq I\} \quad \text{and} \quad
    \min\{\cI(\Gamma)\,:\, \cR(\Gamma) \leq \alpha\}\,,
\end{align}
where $I \in (0, K]$ and $\alpha \in (0, 1)$ are parameters controlling, respectively, the expected size of the prediction set and the allowable miscoverage probability.

At first glance, this problem may appear to fall outside the scope of our framework. However, with a mild modification, it can be naturally incorporated. Specifically, consider a mapping $\bpi = (\pi_1, \ldots, \pi_K) : \cX \mapsto [0, 1]^K$. For any such $\bpi$, define $\Gamma^{\bpi}$ as a set-valued random variable satisfying:
\begin{align*}
\Prob^{\bpi}(y \in \Gamma^{\bpi} \mid \bX) = \pi_y(\bX)\quad\text{and}\quad \left(\Gamma^{\bpi} \independent Y\right) \mid \bX\quad\text{and}\quad(\bX, Y) \sim \Prob\,.
\end{align*}

As before, we denote by $\Prob^{\bpi}$ the joint distribution of $(\bX, Y, \Gamma^{\bpi})$. Since the prediction $\Gamma^{\bpi}$ is fully specified by the mapping $\bpi$, we slightly abuse notation and define the risk and size of $\bpi$ as follows:
\begin{align*}
\cR(\bpi) \eqdef \Prob^\bpi(Y \notin \Gamma^\bpi)\quad\text{and}\quad\cI(\bpi) = \Exp^\bpi |\Gamma^\bpi|.
\end{align*}

Introducing $p_y(\bx) = \Prob(Y = y \mid \bX = \bx)$, we observe that:
\begin{align}
\label{eq:set_valued2}
\cR(\bpi) \eqdef \Exp\left[\sum_{y \in [K]}p_y(\bX)(1 - \pi_y(\bX))\right]\quad\text{and}\quad\cI(\bpi) = \Exp\left[\sum_{y \in [K]}\pi_y(\bX)\right]\,.
\end{align}
The key insight here is that both the risk and size are now linear functionals of $\bpi$, allowing us to apply the same methodology developed earlier for standard classification. The only distinction is that previously we considered mappings $\pi : \cX \to \Delta([K])$, whereas we now work with $\bpi : \cX \mapsto [0, 1]^K$. This adjustment is minimal and can be seamlessly integrated into our framework.

Moreover, this formulation allows for the inclusion of additional linear constraints, depending on the application. For instance, a variant of the churn constraint takes the form:
\begin{align*}
\cC(\bpi) \eqdef \Prob^\bpi(f(\bX) \notin \Gamma^\bpi) \leq \beta,
\end{align*}
for a fixed classifier $f : \cX \mapsto [K]$ and a control parameter $\beta \in (0, 1)$. Again, we find that:
\begin{align*}
\cC(\bpi) = \Exp\left[(1 - \bpi_{f(\bX)}(\bX))\right] = \Exp\left[\sum_{y \in [K]}\ind{f(\bX) = y}(1 - \bpi_{y}(\bX))\right]\,,
\end{align*}
which is linear in $\bpi$ and, hence, can be included in the framework. We do not pursue the detailed theoretical exposition of this case, as the modifications to be made are rather straightforward.





\section{Conclusion}

We presented a unified post-processing framework for multi-class classification under general expectation-based constraints. Our method constructs randomized classifiers by solving a regularized dual optimization problem, allowing for explicit control of constraint satisfaction and statistical guarantees. It operates independently of the training process and can be applied to any base predictor, provided estimators of relevant quantities are available.
The proposed algorithm is theoretically grounded, with finite-sample guarantees that hold under minimal assumptions. It supports a wide range of applications, including fairness constraints, reject options, and churn minimization, and naturally extends to combinations thereof --- see Section~\ref{app:numeric} for some illustrative examples.
By decoupling constrained prediction from model training, our framework offers a flexible and efficient solution for integrating complex requirements into classification systems. Future work may explore extensions to structured prediction tasks and/or alternative regularization schemes.

\bibliography{bibli}

\begin{thebibliography}{41}
\providecommand{\natexlab}[1]{#1}
\providecommand{\url}[1]{\texttt{#1}}
\expandafter\ifx\csname urlstyle\endcsname\relax
  \providecommand{\doi}[1]{doi: #1}\else
  \providecommand{\doi}{doi: \begingroup \urlstyle{rm}\Url}\fi

\bibitem[Agarwal et~al.(2018)Agarwal, Beygelzimer, Dud{\'\i}k, Langford, and
  Wallach]{Agarwal_Beygelzimer_Dubik_Langford_Wallach18}
A.~Agarwal, A.~Beygelzimer, M.~Dud{\'\i}k, J.~Langford, and H.~Wallach.
\newblock A reductions approach to fair classification.
\newblock In \emph{Proceedings of the 35th International Conference on Machine
  Learning}, 2018.

\bibitem[Agarwal et~al.(2019)Agarwal, Dudik, and Wu]{agarwal2019fair}
A.~Agarwal, M.~Dudik, and Z.~S. Wu.
\newblock Fair regression: Quantitative definitions and reduction-based
  algorithms.
\newblock In \emph{International Conference on Machine Learning}, 2019.

\bibitem[Allen-Zhu(2018)]{allenzhu2021make}
Z.~Allen-Zhu.
\newblock How to make the gradients small stochastically: Even faster convex
  and nonconvex sgd.
\newblock \emph{Advances in Neural Information Processing Systems}, 31, 2018.

\bibitem[Audibert and Tsybakov(2007)]{tsybakov2007fast}
J-Y Audibert and AB~Tsybakov.
\newblock Fast learning rates for plug-in classifiers.
\newblock \emph{Annals of Statistics}, 35\penalty0 (2):\penalty0 608--633,
  2007.

\bibitem[Bartlett and Wegkamp(2008)]{bartlett2008classification}
P.~Bartlett and M.~Wegkamp.
\newblock Classification with a reject option using a hinge loss.
\newblock \emph{Journal of Machine Learning Research}, 9\penalty0 (8), 2008.

\bibitem[Boyd and Vandenberghe(2004)]{boyd2004convex}
S.~Boyd and L.~Vandenberghe.
\newblock \emph{Convex optimization}.
\newblock Cambridge university press, 2004.

\bibitem[Cao et~al.(2022)Cao, Cai, Feng, Gu, Gu, An, Niu, and
  Sugiyama]{cao2022generalizing}
Y.~Cao, T.~Cai, L.~Feng, L.~Gu, J.~Gu, B.~An, G.~Niu, and M.~Sugiyama.
\newblock Generalizing consistent multi-class classification with rejection to
  be compatible with arbitrary losses.
\newblock \emph{Advances in neural information processing systems},
  35:\penalty0 521--534, 2022.

\bibitem[Celis et~al.(2019)Celis, Huang, Keswani, and
  Vishnoi]{celis2019classification}
E.~Celis, L.~Huang, V.~Keswani, and N.~Vishnoi.
\newblock Classification with fairness constraints: A meta-algorithm with
  provable guarantees.
\newblock In \emph{Proceedings of the conference on fairness, accountability,
  and transparency}, pages 319--328, 2019.

\bibitem[Charoenphakdee et~al.(2021)Charoenphakdee, Cui, Zhang, and
  Sugiyama]{charoenphakdee2021classification}
N.~Charoenphakdee, Z.~Cui, Y.~Zhang, and M.~Sugiyama.
\newblock Classification with rejection based on cost-sensitive classification.
\newblock In \emph{International Conference on Machine Learning}, pages
  1507--1517. PMLR, 2021.

\bibitem[Chow(1957)]{Chow57}
{C.} Chow.
\newblock An optimum character recognition system using decision functions.
\newblock \emph{IRE Transactions on Electronic Computers}, EC-6\penalty0
  (4):\penalty0 247--254, 1957.

\bibitem[Chzhen et~al.(2020)Chzhen, Denis, Hebiri, Oneto, and
  Pontil]{chzhen2020-discr}
E.~Chzhen, C.~Denis, M.~Hebiri, L.~Oneto, and M.~Pontil.
\newblock Fair regression via plug-in estimator and recalibration.
\newblock \emph{NeurIPS 2020}, 2020.

\bibitem[Cotter et~al.(2019)Cotter, Jiang, Gupta, Wang, Narayan, You, and
  Sridharan]{cotter2019optimization}
A.~Cotter, H.~Jiang, M.~Gupta, S.~Wang, T.~Narayan, S.~You, and K.~Sridharan.
\newblock Optimization with non-differentiable constraints with applications to
  fairness, recall, churn, and other goals.
\newblock \emph{Journal of Machine Learning Research}, 20\penalty0
  (172):\penalty0 1--59, 2019.

\bibitem[Denis and Hebiri(2017)]{Denis2017}
C.~Denis and M.~Hebiri.
\newblock Confidence sets with expected sizes for multiclass classification.
\newblock \emph{Journal of Machine Learning Research}, 18\penalty0
  (102):\penalty0 1--28, 2017.
\newblock URL \url{http://jmlr.org/papers/v18/16-596.html}.

\bibitem[Denis and Hebiri(2020)]{denis2020consistency}
C.~Denis and M.~Hebiri.
\newblock Consistency of plug-in confidence sets for classification in
  semi-supervised learning.
\newblock \emph{Journal of Nonparametric Statistics}, 32\penalty0 (1):\penalty0
  42--72, 2020.

\bibitem[Denis et~al.(2024)Denis, Elie, Hebiri, and Hu]{denis2021}
C.~Denis, R.~Elie, M.~Hebiri, and F.~Hu.
\newblock Fairness guarantees in multi-class classification with demographic
  parity.
\newblock \emph{Journal of Machine Learning Research}, 25\penalty0
  (130):\penalty0 1--46, 2024.

\bibitem[Foster et~al.(2019)Foster, Sekhari, Shamir, Srebro, Sridharan, and
  Woodworth]{foster2019complexity}
D.~Foster, A.~Sekhari, O.~Shamir, N.~Srebro, K.~Sridharan, and B.~Woodworth.
\newblock The complexity of making the gradient small in stochastic convex
  optimization.
\newblock In \emph{Conference on Learning Theory}, pages 1319--1345. PMLR,
  2019.

\bibitem[Gao and Pavel(2017)]{gao2017}
B.~Gao and L.~Pavel.
\newblock On the properties of the softmax function with application in game
  theory and reinforcement learning.
\newblock \emph{arXiv preprint arXiv:1704.00805}, 2017.

\bibitem[Gaucher et~al.(2023)Gaucher, Schreuder, and
  Chzhen]{pmlr-v206-gaucher23a}
S.~Gaucher, N.~Schreuder, and E.~Chzhen.
\newblock Fair learning with wasserstein barycenters for non-decomposable
  performance measures.
\newblock In \emph{Proceedings of The 26th International Conference on
  Artificial Intelligence and Statistics}, volume 206 of \emph{Proceedings of
  Machine Learning Research}, pages 2436--2459. PMLR, 25--27 Apr 2023.
\newblock URL \url{https://proceedings.mlr.press/v206/gaucher23a.html}.

\bibitem[Herbei and Wegkamp(2006)]{herbei2006classification}
R.~Herbei and M.~Wegkamp.
\newblock Classification with reject option.
\newblock \emph{The Canadian Journal of Statistics/La Revue Canadienne de
  Statistique}, pages 709--721, 2006.

\bibitem[Hou and Zhang(2024)]{hou2024finite}
X.~Hou and L.~Zhang.
\newblock Finite-sample and distribution-free fair classification: Optimal
  trade-off between excess risk and fairness, and the cost of group-blindness.
\newblock \emph{arXiv preprint arXiv:2410.16477}, 2024.

\bibitem[Lei(2014)]{lei2014}
J.~Lei.
\newblock Classification with confidence.
\newblock \emph{Biometrika}, 2014.

\bibitem[Lei et~al.(2018)Lei, G'Sell, Rinaldo, Tibshirani, and
  Wasserman]{lei2016}
J.~Lei, M.~G'Sell, A.~Rinaldo, R.~Tibshirani, and L.~Wasserman.
\newblock Distribution-free predictive inference for regression.
\newblock \emph{J. Amer. Statist. Assoc.}, 113\penalty0 (523):\penalty0
  1094--1111, 2018.

\bibitem[Lewis(2023)]{lewis2023}
A.~Lewis.
\newblock A top nine list: Most popular induced matrix norms, 2023.
\newblock URL \url{https://arxiv.org/abs/2309.07190}.

\bibitem[Lichman(2013)]{lichman2013}
M.~Lichman.
\newblock Adult.
\newblock UCI Machine Learning Repository, 2013.
\newblock URL \url{http://archive.ics.uci.edu/ml}.

\bibitem[Luenberger(1997)]{luenberger1997optimization}
D.~Luenberger.
\newblock \emph{Optimization by vector space methods}.
\newblock John Wiley \& Sons, 1997.

\bibitem[Narasimhan(2018)]{narasimhan18a}
H.~Narasimhan.
\newblock Learning with complex loss functions and constraints.
\newblock In Amos Storkey and Fernando Perez-Cruz, editors, \emph{Proceedings
  of the Twenty-First International Conference on Artificial Intelligence and
  Statistics}, volume~84 of \emph{Proceedings of Machine Learning Research},
  pages 1646--1654. PMLR, 09--11 Apr 2018.
\newblock URL \url{https://proceedings.mlr.press/v84/narasimhan18a.html}.

\bibitem[Narasimhan et~al.(2020)Narasimhan, Cotter, Gupta, and
  Wang]{narasimhan2020pairwise}
H.~Narasimhan, A.~Cotter, M.~Gupta, and S.~Wang.
\newblock Pairwise fairness for ranking and regression.
\newblock In \emph{Proceedings of the AAAI Conference on Artificial
  Intelligence}, volume~34, pages 5248--5255, 2020.

\bibitem[Narasimhan et~al.(2024)Narasimhan, Ramaswamy, Tavker, Khurana,
  Netrapalli, and Agarwal]{narasimhan2024consistent}
H.~Narasimhan, H.~G Ramaswamy, S.~Tavker, D.~Khurana, P.~Netrapalli, and
  S.~Agarwal.
\newblock Consistent multiclass algorithms for complex metrics and constraints.
\newblock \emph{Journal of Machine Learning Research}, 25\penalty0
  (367):\penalty0 1--81, 2024.

\bibitem[Nesterov(2005)]{nesterov2005}
Y.~Nesterov.
\newblock Smooth minimization of non-smooth functions.
\newblock \emph{Mathematical programming}, 103\penalty0 (3):\penalty0 127--152,
  2005.

\bibitem[Oneto et~al.(2020)Oneto, Donini, and Pontil]{oneto2020general}
L.~Oneto, M.~Donini, and M.~Pontil.
\newblock General fair empirical risk minimization.
\newblock In \emph{2020 International Joint Conference on Neural Networks
  (IJCNN)}, pages 1--8. IEEE, 2020.

\bibitem[Peyr{\'e} and Cuturi(2019)]{peyre2019computational}
G.~Peyr{\'e} and M.~Cuturi.
\newblock Computational optimal transport: With applications to data science.
\newblock \emph{Foundations and Trends{\textregistered} in Machine Learning},
  11\penalty0 (5-6):\penalty0 355--607, 2019.

\bibitem[Redmond(2009)]{misc_communities_and_crime_183}
M.~Redmond.
\newblock {Communities and Crime}.
\newblock UCI Machine Learning Repository, 2009.
\newblock {DOI}: https://doi.org/10.24432/C53W3X.

\bibitem[Sadinle et~al.(2019)Sadinle, Lei, and Wasserman]{sadinle2019least}
M.~Sadinle, J.~Lei, and L.~Wasserman.
\newblock Least ambiguous set-valued classifiers with bounded error levels.
\newblock \emph{Journal of the American Statistical Association}, 114\penalty0
  (525):\penalty0 223--234, 2019.

\bibitem[Schreuder and Chzhen(2021)]{schreuder2021}
N.~Schreuder and E.~Chzhen.
\newblock Classification with abstention but without disparities.
\newblock In \emph{Uncertainty in artificial intelligence}, pages 1227--1236.
  PMLR, 2021.

\bibitem[Taturyan et~al.(2024)Taturyan, Chzhen, and
  Hebiri]{NEURIPS2024_d5c3ecf3}
G.~Taturyan, E.~Chzhen, and M.~Hebiri.
\newblock Regression under demographic parity constraints via unlabeled
  post-processing.
\newblock In \emph{Advances in Neural Information Processing Systems},
  volume~37, pages 117917--117953, 2024.

\bibitem[Tsybakov(2009)]{tsybakov2009nonparametric}
A.~Tsybakov.
\newblock Nonparametric estimators.
\newblock \emph{Introduction to Nonparametric Estimation}, pages 1--76, 2009.

\bibitem[van~de Geer(2008)]{van2008high}
Sara~A van~de Geer.
\newblock High-dimensional generalized linear models and the lasso.
\newblock \emph{The Annals of Statistics}, 36\penalty0 (2):\penalty0 614, 2008.

\bibitem[Wegkamp(2007)]{wegkamp2007lasso}
M.~Wegkamp.
\newblock Lasso type classifiers with a reject option.
\newblock \emph{Electronic Journal of Statistics}, 1:\penalty0 155--168, 2007.

\bibitem[Xian et~al.(2023)Xian, Yin, and Zhao]{xian2023fair}
R.~Xian, L.~Yin, and H.~Zhao.
\newblock Fair and optimal classification via post-processing.
\newblock In \emph{International conference on machine learning}, pages
  37977--38012. PMLR, 2023.

\bibitem[Yuan and Wegkamp(2010)]{yuan2010classification}
M.~Yuan and M.~Wegkamp.
\newblock Classification methods with reject option based on convex risk
  minimization.
\newblock \emph{Journal of Machine Learning Research}, 11\penalty0 (1), 2010.

\bibitem[Zeng et~al.(2022)Zeng, Dobriban, and Cheng]{zeng2022fair}
X.~Zeng, E.~Dobriban, and G.~Cheng.
\newblock Fair bayes-optimal classifiers under predictive parity.
\newblock \emph{Advances in Neural Information Processing Systems},
  35:\penalty0 27692--27705, 2022.

\end{thebibliography}
\bibliographystyle{plainnat}

\newpage
\appendix


\section{Proofs for results in Section~\ref{sec:methodology}}
\label{app:ProofSection3}
 Let us also define a vector $\bc(\cdot \mid \cdot) \eqdef (c_j(\cdot \mid \cdot))_{j \in [M]}$.  

First we explicit the first order optimality conditions for the problem in~\eqref{min-lse}.

\begin{lemma}
    \label{lem:KKT}
    Let $\blambda^\star \geq 0$ be any minimizer of~\eqref{min-lse} and $\pi^\star = \pi_{\blambda^\star}$ be defined in~\eqref{eq:optimal_entropic}. Then, there exist $\bgamma = (\gamma_{j})_{j \in [M]}$---element-wise non-negative vector such that 
\begin{align}
    \label{kkt}
    \begin{cases}
        &\Exp_\bX\left[\sum_{a \in \cA}c_j(\bX, a)\pi^\star(a \mid \bX)\right] = - \gamma_j \\
        & \gamma_{j}\lambda_{j}^{\star}=0 \\
    \end{cases} \qquad \forall j \in [M]\enspace.
\end{align}
\end{lemma}
\begin{proof}
We first observe that the optimization problem in~\eqref{min-lse} is convex and smooth.
Thus, Karush–Kuhn–Tucker conditions are sufficient for optimally. Furthermore, since Slater's condition is satisfied due to Assumption~\ref{ass:slater}, the latter is also necessary, as the strong duality holds.
In particular, there exist $\bgamma = (\gamma_{j})_{j \in [M]}$---element-wise non-negative matrices such that
\begin{align*}
\begin{cases}
    \nabla_{\blambda} F(\blambda^\star) - \bgamma = \mathbf{0}\\
    \blambda^\star \geq 0\\
    \gamma_{j}\lambda^\star_{j} = 0
\end{cases}
\qquad \forall j \in [M] \enspace.
\end{align*}
To conclude, it is sufficient to evaluate the gradient on $F$, whose expression is given in~\eqref{eq:true_grad} and use the definition of $\pi^\star$.
\end{proof}

\begin{lemma}
    \label{lemma:normF-bd}
    It holds that $\normin{\nabla \hat{F}(\blambda)} \leq \Exp_\bX \pqnormin{1}{2}{\hat{\bC}(\bX)}$.
\end{lemma}
\begin{proof}
Applying Jensen's inequality we get 
\begin{align*}
    \normin{\nabla \hat{F} (\blambda)} = \normin{\Exp_\bX [g({\blambda};\bX)]} \leq \Exp_\bX\normin{g({\blambda};\bX)} \enspace.
\end{align*}
From the definition of $g(\cdot; \bX)$ given in \eqref{eq:stoch_grad} we have
\begin{align*}
    \Exp_\bX\norm{g({\blambda};\bX)}
    &= \Exp_\bX \normin{\hat{\bC}(\bX)\bsigma\big(\beta\big(
    -\hat{\bL}(\bX) - \hat{\bC}(\bX)^\top\blambda\big)\big)}
    \leq \Exp_\bX \pqnormin{1}{2}{\hat{\bC}(\bX)} \enspace,
\end{align*}
where the last inequality follows from the triangle inequality, the definition of $\pqnorm{1}{2}{\cdot}$ and the fact that $\norm{\bsigma(\cdot)}_1=1$.
\end{proof}

\section{Auxilliary results}
\label{app:aux}

In this appendix, we collect some standard auxiliary results, that are used to derive main claims of the paper.

\begin{lemma}[\citet{gao2017}]
\label{lemma:lse} 
Let $\ba=(a_1,\cdots,a_m)$ and $\beta>0$. Define log-sum-exp and softmax functions respectively as
\begin{align*}
    \lse(\boldsymbol{a})  \eqdef  \frac{1}{\beta}\log\left(\sum_{i = 1}^m \exp(\beta a_i)\right) \text{ and }
    \sigma_j(\beta \boldsymbol{a})  \eqdef  \frac{\exp(\beta a_j)}{\sum_{i = 1}^m \exp(\beta a_i)}\quad j \in [m] \enspace.
\end{align*}
The LSE property is as follows
\begin{align*}
    \max\{a_1,\cdots,a_m\} \leq \lse(\boldsymbol{a}) \leq \max\{a_1,\cdots,a_m\} + \frac{\log(m)}{\beta} \enspace.
\end{align*}
Moreover,
\begin{align*}
    \bsigma(\beta \ba) = \nabla\lse(\ba) \text{ is } \beta-\text{Lipschitz and  } \nabla^2 \lse(\ba) = \beta(\diag(\bsigma(\beta\ba))-\bsigma(\beta\ba)\bsigma(\beta\ba)^\top)\enspace.
\end{align*}
\end{lemma}
\begin{lemma}[\citet{boyd2004convex}]
    \label{lemma:lse_variational}
  It holds that
    \begin{align*}
        \lse(\bw) = \max_{\bp \in \Delta}\left\{\scalar{\bw}{\bp} - \frac{1}{\beta} \Psi(\bp)\right\}\,,
    \end{align*}
    where $\Delta$ is the probability simplex in $\bbR^m$ and $\Psi(\bp) = \sum_{i = 1}^m p_i \log(p_i)$. Furthermore, $-\Psi(\bp) \in [0, \log(m)]$ and the optimum in the above optimization problem is achieved at $\bp^{\star} = \sigma(\beta \bw)$.
\end{lemma}
\begin{lemma}[\citet{lewis2023}]
    \label{lemma:norm-12} 
    Let $\bA\in\bbR^{n\times m}$ be a matrix with columns $\ba_i \in \bbR^n$, for $i=1,\cdots,m$. Then, $\pqnorm{1}{2}{\bA}=\max_i \norm{\ba_i}_2$.
\end{lemma}
\begin{lemma}
    \label{lemma:normHessian}
    Let $\bA\in\bbR^{n\times m}$ be a matrix with columns $\ba_i \in \bbR^n$, for $i=1,\cdots,m$, and $\by \in \bbR^m$ be a vector \st $\sum_{i=1}^m y_i = 1$ and $y_i\geq0$ for $i=1,\cdots,n$. Then
    \begin{align*}
        \opnorm{\bA(\diag(\by)-\by\by^\top)\bA^\top} \leq 2 \pqnorm{1}{2}{\bA}^2\enspace.
    \end{align*}
\end{lemma}
\begin{proof}
Triangle inequality yields
\begin{align*}
    \opnorm{\bA(\diag(\by)-\by\by^\top)\bA^\top} \leq \opnorm{\bA\diag(\by)\bA^\top}+\opnorm{\bA\by\by^\top\bA^\top} \enspace.
\end{align*}
It remains to bound each term of the right side of the above inequality. For the first term, we have that $\bA\diag(\by)\bA^\top = \sum_{i=1}^m  y_i \ba_i \ba_i^\top$, so
\begin{align*}
    \opnorm{\bA\diag(\by)\bA^\top} &= \opnorm{\sum_{i=1}^m  y_i \ba_i \ba_i^\top} \leq \sum_{i=1}^m y_i \opnorm{\ba_i \ba_i^\top} = \sum_{i=1}^m y_i \norm{\ba_i}_2^2 \\ 
    &\leq \max_{i = 1, \ldots, m} \norm{\ba_i}_2^2 = \pqnorm{1}{2}{\bA}^2\enspace,
\end{align*}
where the first inequality follows from the fact that $\norm{\by}_1 =1$ and the last one follows from Lemma~\ref{lemma:norm-12}.
For the second term, we have
\begin{align*}
    \opnorm{\bA\by\by^\top\bA^\top} = \norm{\bA\by}^2_2 \leq \pqnorm{1}{2}{\bA
    }^2\enspace,
\end{align*}
where the last inequality follows from the fact that $\norm{\by}_1 =1$ and the definition of $\pqnorm{1}{2}{\cdot}$. The proof is concluded.
\end{proof}

\section{Numerical illustrations}
\label{app:numeric}

In this appendix, we provide some basic illustrations of our procedure. Our methodology can handle a wide range of losses and constraints and the goal here is to highlight its flexibility through two examples: fairness and reject option frameworks~\footnote{The code is available at \url{https://github.com/taturyan/post-processing}.}.
In particular, we implemented our \texttt{COPT} Algorithm~\ref{alg:main} with the correct form of the loss and the constraint.

\subsection{Data} 
We conducted our experiments on 2 datasets:
\begin{enumerate}
    \item \textit{Communities and Crime} dataset (\citet{misc_communities_and_crime_183}) contains socio-economic, law enforcement, and crime data about communities in the US with 1994 examples. We focus on predicting the number of violent crimes per $10^5$ population within the range of $[0,1]$. In case of fairness, we consider ethnicity as protected variable.
    \item \textit{Adult} dataset~(\citet{lichman2013}) consists in $48842$ instances of socio-economic data. The goal is to predict whether the annual income of an individual exceeds 50K/year based on census data. We clean and preprocess it, and use a smaller sub-sample of 2000 points with 8 features throughout our experiments.
When considering the fairness problem, we consider the sex as a sensitive attribute, distinguishing between male and female individuals. 
\end{enumerate}

\subsection{Experiments}
The general pipeline of the algorithm is the same for all constraints, with slight difference in the case of fairness constraints, where sensitive attributes have to be taken into account. 
\subsubsection{Controlled rejection.}
We randomly split the data into training, unlabeled and testing  sets with proportions of $0.4\times0.4\times0.2$. We use $\data_{\trainlabeled} = \{(\bx_i ,y_i)_{i = 1}^n\}$ to train a base regressor to estimate $p_y$ for all $y\in \mathcal{A} = [K]$. We use simple \textit{LogisticRegression} from \textit{scikit-learn} for training the base classifier. Finally, we use the trained classifier to train the Algorithm~\ref{alg:main} with $\data_{\trainunlab} = (\bx)_{i = n+1}^{n+T}$ for $N$ (note that our theory suggests that $N = T$ is enough, but we have noticed that larger $N$ can be more beneficial in practice) iterations. 
We use $\data_{\test} = \{(\bx'_i, y_i')_{i = 1}^m\}$ to collect evaluation statistics. 
We take the sets of thresholds $\mathcal{I}=\{0.2, 0.1, 0.05, 0.025, 0.0125\}$ for the rejection control. We set  $\beta=0.5\sqrt{T}\log{\sqrt{T}}$ and repeat the pipeline 10 times.

\paragraph{Evaluation measures.} For reject option, we use $\data_{\test} $ to collect the following statistics of any (randomized) prediction $\pi$:
\begin{equation*}
    \hat{\risk}(\pi) \eqdef \frac{1}{m}\sum_{i = 1}^m \sum_{\hat{y}} \ind{\hat{y} \neq y'} \pi(\hat y \mid \bx_i'), \qquad\qquad 
    \hat{C}(\pi) \eqdef \frac{1}{m}\sum_{i = 1}^m \pi (r \mid \bx_i') \enspace,
\end{equation*} 
which correspond to the empirical risk and the control on the rejection respectively.

\paragraph{Results.} In Figure~\ref{fig:RO} we illustrate the comparison of our model with the model of \citet{denis2020consistency} -- denoted by \texttt{BinClassRO} -- on \textit{Communities and Crime} and \textit{Adult} datasets. We consider 3 variants of our procedure, one with $T=100$, the second with $T=800$, and an alternative approach where $T=100$ and we output from the probabilistic prediction the argmax of the probabilities as prediction. (This is a more conventional deterministic classifier for which our theory does not hold.) 

Notably, our methods perform quite similarly to the baseline method even though our method is not tailored only to the specific task of rejection option. We also observe that increase $T=800$ has only a small positive impact in this experiments.

\begin{figure}[!htb]
\minipage{0.48\textwidth}
\centering
\includegraphics[width=\linewidth]{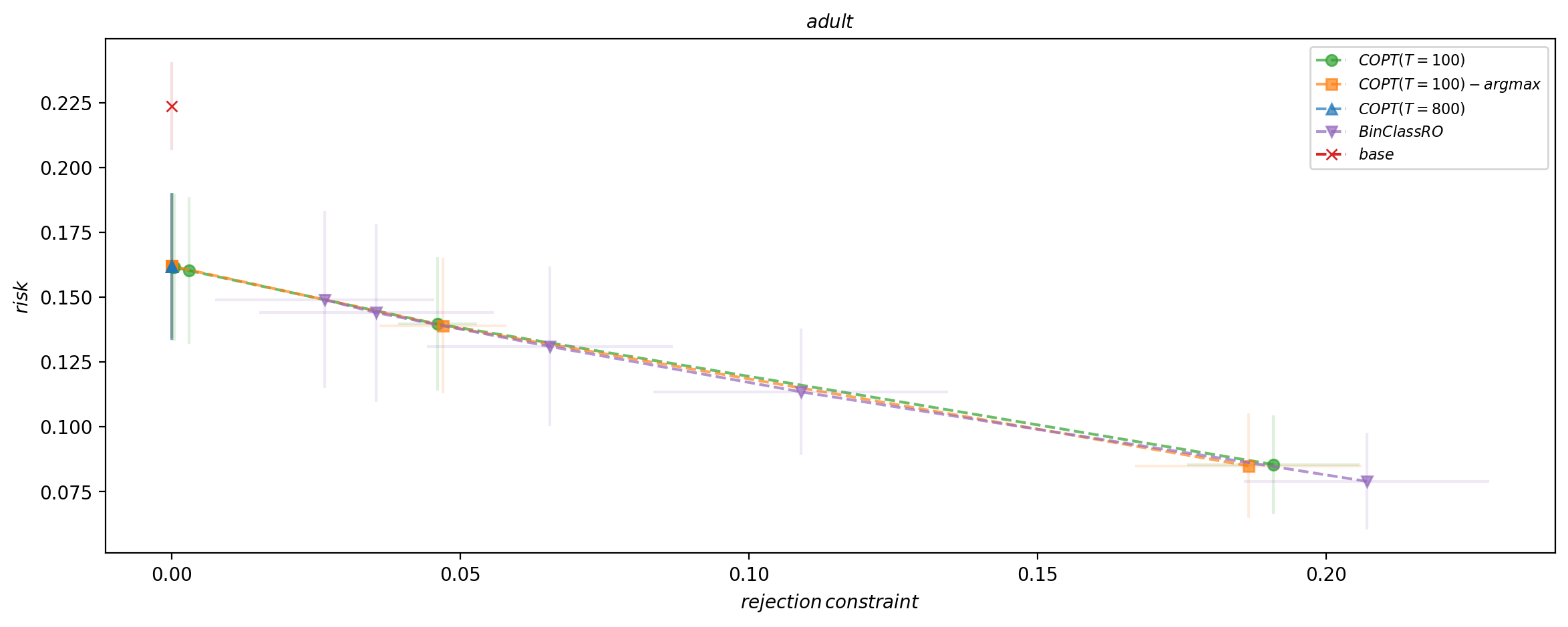}
\endminipage 
\minipage{0.48\textwidth}
\centering
\includegraphics[width=\linewidth]{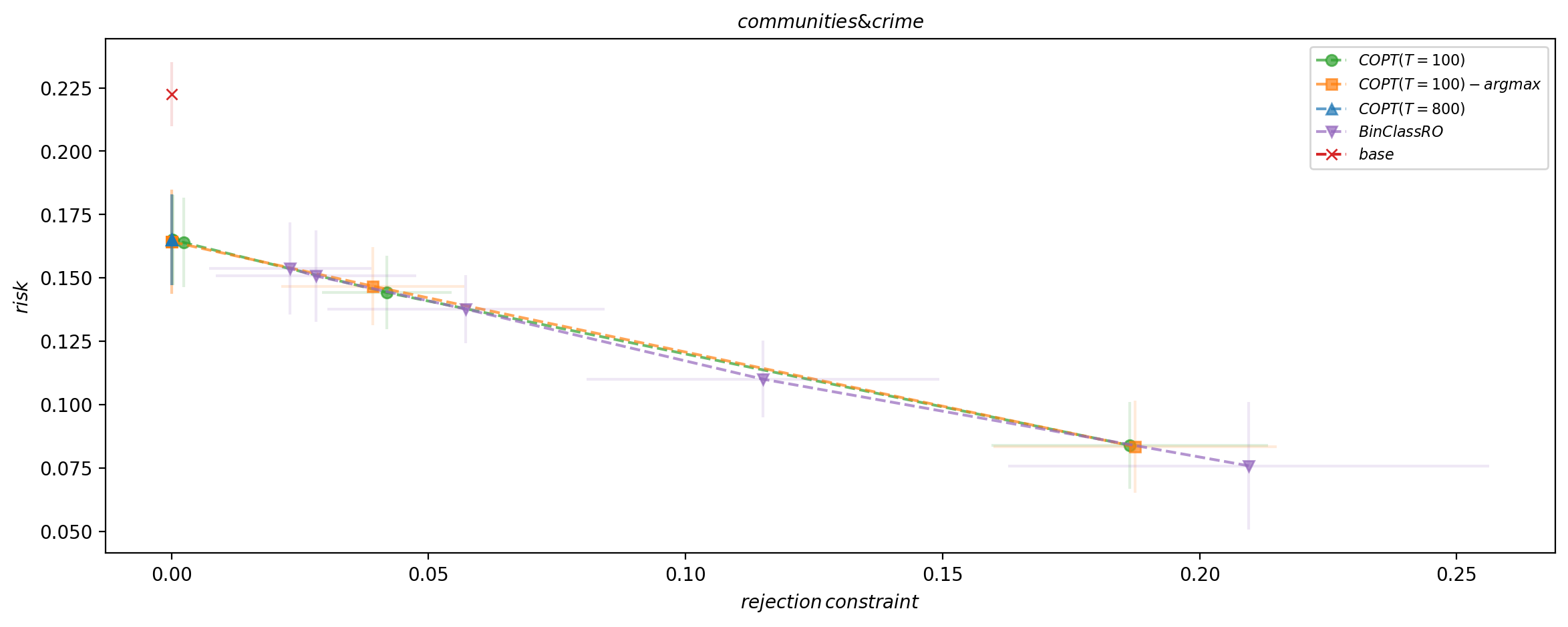}
\endminipage
\caption{Control of the rejection w.r.t. risk}
\label{fig:RO}
\end{figure}

\subsubsection{Unaware Demographic Parity.}
We randomly split the data into training, unlabeled and testing  sets with proportions of $0.4\times0.4\times0.2$. We use $\data_{\trainlabeled} = \{(\bx_i ,y_i)_{i = 1}^n\}$ to train a base classifier to estimate $p_y$ for all $y\in \mathcal{A} = [K]$ and $\{(\bx_i ,s_i)_{i = 1}^n\}$ to train a classifier to estimate $\tau$. We use simple \textit{LogisticRegression}s from \textit{scikit-learn} for training both classifiers. Finally, we use the trained classifiers to train the Algorithm~\ref{alg:main} with $\data_{\trainunlab} = (\bx)_{i = n+1}^{n+T}$ for $N$ (note that our theory suggests that $N = T$ is enough, but we have noticed that larger $N$ can be more beneficial in practice) iterations. 
We use $\data_{\test} = \{(\bx'_i, s'_i, y_i')_{i = 1}^m\}$ to collect evaluation statistics. 
We take the sets of fairness violation thresholds $\{(2^{-i},2^{-i})_{i \in \mathcal{I}}\}$, where $\mathcal{I}=\{1, 2, 4, 5, 6, 8, 16, 32, 128\}$ for \textit{Adult} and $\mathcal{I}=\{1.3, 2, 4, 8, 16, 32, 512\}$ for \textit{Communities and Crime}. We set  $\beta=0.5\sqrt{T}\log{\sqrt{T}}$ and repeat the pipeline 10 times.

\subsection{Evaluation measures.} We use $\data_{\test} $ to collect the following statistics of any (randomized) prediction $\pi$ 
\begin{align*}
    &\hat{\risk}(\pi) \eqdef \frac{1}{m}\sum_{i = 1}^m \sum_{\hat{y}} \ind{\hat{y} \neq y'} \pi(\hat y \mid \bx_i')\enspace,\\
    &\hat{U}_s(\pi) \eqdef \left|\frac{1}{m_s}\sum_{i = 1}^m \sum_{\hat{y}} \pi(\hat y \mid \bx_i')\ind{s_i' = s} - \frac{1}{m} \sum_{\hat{y}} \sum_{i = 1}^m \pi(\hat y \mid \bx_i')\right|\enspace,
\end{align*} 
which correspond to the empirical risk and the empirical group-wise unfairness quantified by the Kolmogorov-Smirnov distance of a randomized prediction.

\paragraph{Results} In Figure~\ref{fig:fair} we illustrate a comparison of our model with \textit{fairlearn} python package on \textit{Communities and Crime} and \textit{Adult} datasets. We considered two versions of this algorithm: \texttt{fairlearn - 1}where we use to whole dataset $\data_{\trainlabeled}$ and $\data_{\trainunlab}$ for the estimation of the $p_y$ and \texttt{fairlearn - 2} where we only use the sample $\data_{\trainlabeled}$ for that purpose.

Here again, our proposal adapts very well to the problem even though it is not only devoted to the fairness problem with demographic parity (DP) constraint.

\begin{figure}[!htb]
\minipage{0.48\textwidth}
\centering
\includegraphics[width=\linewidth]{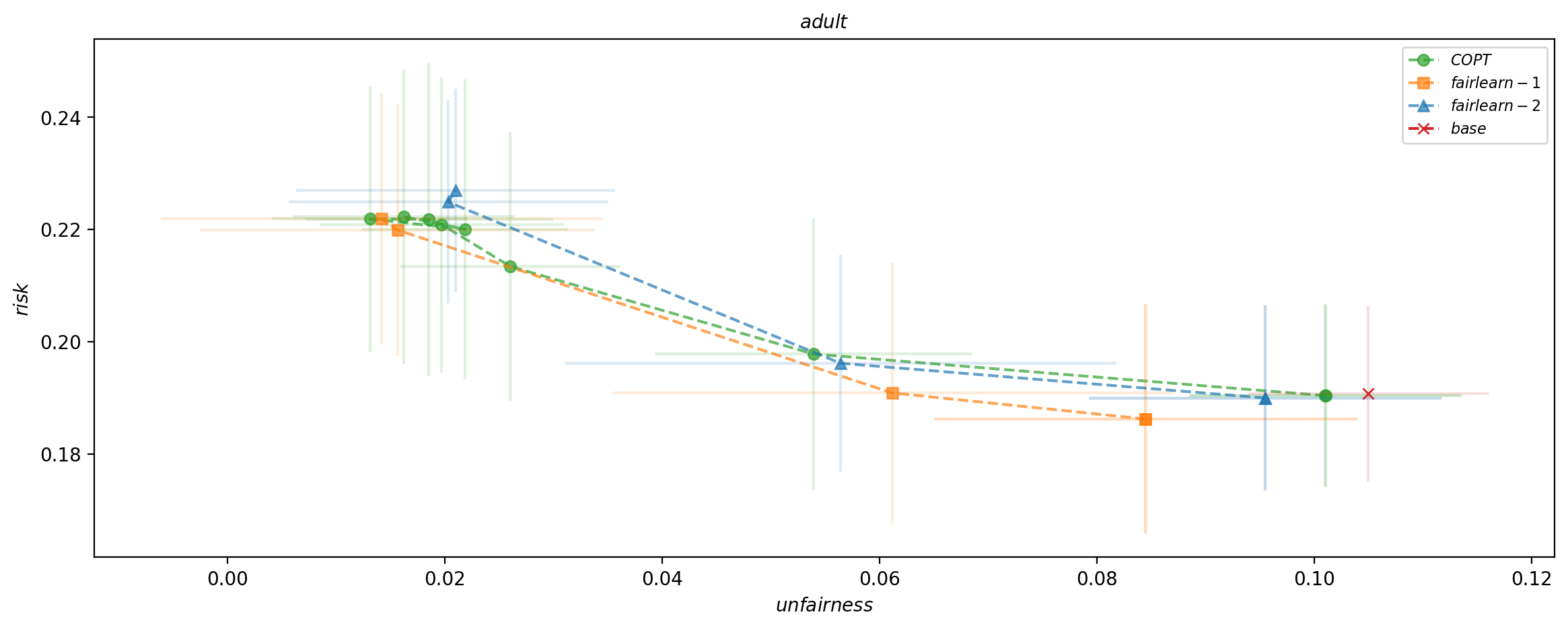}
\endminipage 
\minipage{0.48\textwidth}
\centering
\includegraphics[width=\linewidth]{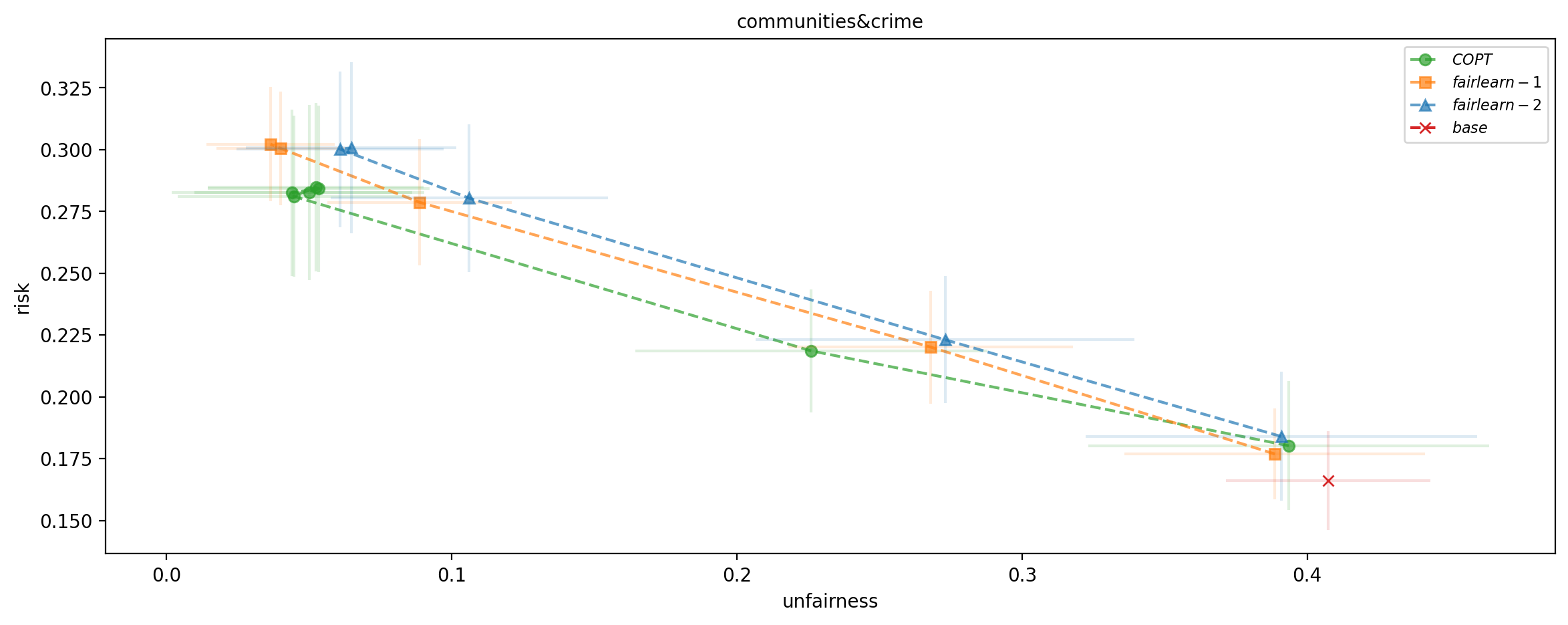}
\endminipage
\caption{Violation of DP fairness constraint w.r.t. risk}
\label{fig:fair}
\end{figure}


\end{document}